\documentclass[11pt,reqno]{amsart}
\usepackage{graphicx}
\usepackage{enumitem}
\setlist[itemize]{leftmargin=2.5em}
\setlist[enumerate]{font={\upshape}, label=\arabic*., leftmargin=2.5em}
\usepackage{mathtools}
\usepackage{amssymb}

\usepackage[pdftex
]{hyperref}
\hypersetup{
	colorlinks=true,
	allcolors=blue,
}
\usepackage{cite}
\usepackage[capitalize]{cleveref}
\usepackage[margin=0.8in]{geometry}
\newtheorem{lemma}{Lemma}[section]
\newtheorem{claim}[lemma]{Claim}
\Crefname{claim}{Claim}{Claims}
\newtheorem{theorem}[lemma]{Theorem}
\newtheorem{corollary}[lemma]{Corollary}
\expandafter\let\expandafter\oldproof\csname\string\proof\endcsname
\let\oldendproof\endproof
\renewenvironment{proof}[1][\proofname]{%
	\oldproof[\normalfont\bfseries #1]%
}{\oldendproof}
\newenvironment{subproof}[1][\normalfont\it\proofname]{%
	\begin{proof}[#1]%
	}{%
	\end{proof}%
}
\newcommand{\dd}{\textquotedblleft}
\newcommand{\ee}{\textquotedblright}
\newcommand{\mac}{\mathcal}

\newcommand{\vep}{\varepsilon}
\newcommand{\nin}{\notin}
\renewcommand{\subset}{\subseteq}
\renewcommand{\supset}{\supseteq}
\newcommand{\edge}{\operatorname{\mathsf{e}}}
\DeclarePairedDelimiter\abs{\lvert}{\rvert}%
\DeclarePairedDelimiter\ceil{\lceil}{\rceil}%
\DeclarePairedDelimiter\floor{\lfloor}{\rfloor}%
\makeatletter
\def\old@comma{,}
\catcode`\,=13
\def,{%
	\ifmmode%
	\old@comma\discretionary{}{}{}%
	\else%
	\old@comma%
	\fi%
}
\makeatother
\makeatletter
\newcommand{\leqnomode}{\tagsleft@true}
\newcommand{\reqnomode}{\tagsleft@false}
\makeatother
\begin{document}	
	\title{Linear-sized minors with given edge density}
	\author{Tung H. Nguyen}
	\address{Princeton University, Princeton, NJ 08544, USA}
	\email{\href{mailto:tunghn@math.princeton.edu}
		{tunghn@math.princeton.edu}}
	\thanks{Partially supported by AFOSR grants A9550-10-1-0187 and FA9550-22-1-0234, and NSF grant DMS-2154169.}
	\begin{abstract}
		It is proved that 
		for every $\vep>0$,
		there exists $K>0$ such that
		for every integer $t\ge2$,
		every graph with chromatic number at least 
		$Kt$ contains a minor with $t$ vertices and edge density at least $1-\vep$.
		Indeed, building on recent work of Delcourt and Postle on linear Hadwiger's conjecture, for $\vep\in(0,\frac{1}{256})$
		we can take $K=C\log\log(1/\vep)$ where $C>0$ is a universal constant,
		which extends their recent $O(t\log\log t)$ bound on the chromatic number of graphs with no $K_t$ minor.
	\end{abstract}
	\maketitle
	\section{Introduction}
	All graphs in this paper are finite and with no loops or parallel edges.
	Given a graph $G$,
	a \emph{minor} of $G$ is a graph obtained from a subgraph of $G$ by a sequence of edge contractions;
	and the {\em chromatic number} of $G$, denoted by $\chi(G)$,
	is the least integer $n\ge0$ such that the vertices of $G$ can be colored by $n$ colors so that no two adjacent vertices get the same color.
	
	For integer $t\ge2$, let $K_t$ denote the complete graph on $t$ vertices, and let $h(t)$ be the least positive integer such that
	every graph with chromatic number at least $h(t)$ contains a $K_t$ minor.
	Hadwiger's well-known conjecture~\cite{MR12237} from 1943 states that $h(t)\le t$ for all $t\ge2$.
	This is a vast generalization of the four-color theorem and is wide open for $t\ge7$.
	(See the survey~\cite{MR3526944} for more details.)
	Linear Hadwiger's conjecture~\cite{MR1654153}, one of its weakenings, is still open and asks for the existence of a universal constant $C>0$ such that $h(t)\le Ct$.
	
	Since every graph with chromatic number at least $n\ge1$ has a subgraph with minimum degree at least $n-1$,
	Mader~\cite{MR229550} initiated the approach via the average degree and
	proved that for some constant $C>0$, every graph with average degree at least $Ct\log t$ contains a $K_t$ minor\footnote{In this paper $\log$ denotes the natural logarithm.}.
	Kostochka~\cite{MR713722,MR779891} and Thomason~\cite{MR735367} independently improved this to $Ct\sqrt{\log t}$,
	which is optimal up to the constant factor $C$ by random graph constructions.
	Thus $h(t)\le O(t\sqrt{\log t})$,
	which remained the record for nearly four decades
	until Norin, Postle, and Song~\cite{nps2019} showed that $h(t)\le C_{\vep}t(\log t)^{1/4+\vep}$ for every $\vep>0$ and $t\ge2$,
	where $C_{\vep}>0$ is a constant depending on $\vep$.
	Following subsequent improvements,
	Delcourt and Postle~\cite{delcourt2021reducing} very recently achieved the best known bound $h(t)\le O(t\log\log t)$.
	
	Another natural relaxation for Hadwiger's conjecture seeks to optimize the edge density of a minor on $t$ vertices of every graph with chromatic number at least $t$.
	By a theorem of Mader~\cite{MR229550}, such a graph contains a minor with at most $t$ vertices and minimum degree at least $t/2$,
	and hence contains a minor on $t$ vertices and with at least $\frac{1}{4}{t\choose2}$ edges.
	Norin and Seymour~\cite{nst2022} very recently showed that for all sufficiently large integers $n$, every graph on $n$ vertices and with stability number two has a minor with $\ceil{n/2}$ vertices and fewer than $1/76$ of all possible edges missing\footnote{By a theorem of Plummer, Stiebitz, and Toft~\cite{MR2070161}, Hadwiger's conjecture for graphs with stability number two is equivalent to the statement that every graph with $n$ vertices and stability number two has a $K_{\ceil{n/2}}$ minor.}.
	
	The main aim of this paper is to provide a similar relaxation for linear Hadwiger's conjecture.
	For $\vep>0$ and an integer $t\ge2$,
	an \emph{$(\vep,t)$-dense} graph is a graph with $t$ vertices and at least $(1-\vep){t\choose2}$ edges.
	Thus every $(\vep,t)$-dense graph is a copy of $K_t$ if $\vep\le t^{-2}$;
	and by averaging, every $(\vep,n)$-dense graph contains an $(\vep,t)$-dense subgraph for every $n\ge t\ge2$.
	Here is our main result.
	\begin{theorem}
		\label{thm:chi}
		There is an integer $C>0$ such that for every $\vep\in(0,\frac{1}{256})$ and every integer $t\ge2$,
		every graph with chromatic number at least $Ct\log\log(1/\vep)$ contains an $(\vep,t)$-dense minor.
	\end{theorem}
	\cref{thm:chi} has two implications:
	first, given any number strictly smaller than $1$, one can find in a graph a minor whose edge density is at least that number
	and whose size is linear in terms of the chromatic number of the host graph;
	second, by letting $\vep=t^{-2}$ for $t$ large,
	one can recover the $O(t\log\log t)$ bound of Delcourt and Postle.
	In fact, our proof of \cref{thm:chi} is a modification of their argument,
	illustrating that the \dd redundant\ee{} $\log$ factor can be interpreted as a quantity depending on the edge density of the desired minors instead of the number of vertices.
	Along the way, we shall see in the next sections that this phenomenon, in fact, also holds for several known results in the area.
	For example,
	we shall prove that in order to color graphs with no $(\vep,t)$-dense minors,
	it is enough to color those with not too many vertices;
	this is similar to~\cite[Theorem 1.6]{delcourt2021reducing}
	but with $1/\vep$ inside all of the $\log$ terms.
	\begin{theorem}
		\label{thm:main}
		There is an integer $C_{\ref{thm:main}}>0$ such that the following holds.
		Let $\vep\in(0,\frac{1}{256})$,
		and let $t\ge 4\log^2(1/\vep)$ be an integer.
		For every graph $G$, let\footnote{Here $F\subset G$ means $F$ is a subgraph of $G$.}
		\[f_{\ref{thm:main}}(G,\vep,t):=
		1+\max_{F\subset G}\left\{\frac{\chi(F)}{a}: 
		a\ge\frac{t}{\sqrt{2\log(1/\vep)}},\,\abs{F}\le C_{\ref{thm:main}}a\log^4(1/\vep),\,
		\text{$F$ is $(\vep,a)$-dense-minor-free}\right\}.\]
		If $\chi(G)\ge C_{\ref{thm:main}}t\cdot f_{\ref{thm:main}}(G,\vep,t)$,
		then $G$ contains an $(\vep,t)$-dense minor.
	\end{theorem}
	Let us see how \cref{thm:main} yields \cref{thm:chi}.
	We need a result first proved by Woodall~\cite[Theorem 4]{MR889352}, which also follows from an elegant argument of Duchet and Meyniel~\cite{MR671905} (see~\cite[Theorem 4.2]{MR3526944}).
	\begin{theorem}
		\label{thm:duchetmeyniel}
		For every integer $t\ge2$, every graph $G$ with no $K_t$ minor contains an induced subgraph with chromatic number less than $t$ and at least $\frac{1}{2}\abs{G}$ vertices.
	\end{theorem}
	Here is an immediate corollary of~\cref{thm:duchetmeyniel}.
	\begin{corollary}
		\label{cor:main}
		For every integer $t\ge2$,
		every $K_t$-minor-free graph $G$ with $\abs{G}\ge 3t$ satisfies $\chi(G)\le 4t\log(\abs{G}/t)$.
	\end{corollary}
	\begin{proof}
		Let $q:=\abs{G}/t\ge3$.
		By~\cref{thm:duchetmeyniel}, there is an integer $s\ge0$ and a chain of sets $V(G)=V_0\supset V_1\supset\ldots\supset V_s=\emptyset$ such that for every $i\in[s]$,
		$\abs{V_i}\le\frac{1}{2}\abs{V_{i-1}}$
		and $\chi(G[V_{i-1}\setminus V_{i}])<t$;
		in particular $\abs{V_i}\le 2^{-i}\abs{G}=2^{-i}tq$ for all $i\in[s]$.
		Let $r\ge0$ be maximal with $\abs{V_r}\ge t\log q$;
		then $\abs{V_{r+1}}\le t\log q$ and
		$2^{-r}tq\ge \abs{V_r}\ge t\log q$
		which imply $2^r\log q\le q$
		and so $r\le \log_2 q \le 2\log q$
		(note that $q\ge3$).
		Therefore
		\[\chi(G)\le\sum_{i=1}^{r+1}\chi(G[V_{i-1}\setminus V_i])+\chi(G[V_{r+1}])
		\le (r+1)t+t\log q
		\le 3t\log q+t\log q
		=4t\log q.
		\qedhere\]
	\end{proof}
	The proof of~\cref{thm:chi} now follows shortly.
	\begin{proof}
		[Proof of \cref{thm:chi}, assuming \cref{thm:main}]
		Choose $C$ such that $C\ge17C_{\ref{thm:main}}^2$
		and every graph with average degree at least $\frac{1}{4}Ct\log t$ has a $K_t$ minor. Let $\Gamma:=\log(1/\vep)\ge2$, and
		let $G$ be a graph with $\chi(G)\ge Ct\log\Gamma$.
		If $t\le 4\Gamma^2$,
		then $t\le\Gamma^4$; so $\chi(G)\ge\frac{1}{4}Ct\log t$ and $G$ has a $K_t$ minor.
		If $t\ge 4\Gamma^2$,
		then for every integer $a\ge(2\Gamma)^{-1/2}t$ and every $(\vep,a)$-dense-minor-free subgraph $F$ of $G$ with $\abs{F}\le C_{\ref{thm:main}}\Gamma^4a$,
		\begin{itemize}
			\item if $\abs{F}\le 3C_{\ref{thm:main}}a$ then $\frac{1}{a}\chi(F)\le 3C_{\ref{thm:main}}$; and
			
			\item if $\abs{F}\ge 3C_{\ref{thm:main}}a$
			then $\frac{1}{a}\chi(F)\le 4C_{\ref{thm:main}}\log(\Gamma^4)
			=16C_{\ref{thm:main}}\log\Gamma$
			by \cref{cor:main}
			applied to $F$ and $C_{\ref{thm:main}}a$.
		\end{itemize}
		Thus $f_{\ref{thm:main}}(G,\vep,t)
		\le 1+16C_{\ref{thm:main}}\log\Gamma
		\le 17C_{\ref{thm:main}}\log\Gamma$,
		and so $\chi(G)\ge Ct\log\Gamma\ge C_{\ref{thm:main}}t\cdot f_{\ref{thm:main}}(G,\vep,t)$
		which implies, by \cref{thm:main}, that $G$ contains an $(\vep,t)$-dense minor.
		This proves~\cref{thm:chi}.
	\end{proof}
	The rest of this paper is devoted to proving \cref{thm:main}.
	We make no serious attempt to optimize absolute constants throughout the paper.
	\subsection*{Notation}
	For an integer $n\ge0$, let $[n]$ denote $\{1,2,\ldots,n\}$ if $n\ge1$ and $\emptyset$ if $n=0$.
	For a graph $G$ with vertex set $V(G)$ and $E(G)$,
	let $\abs G:=\abs{V(G)}$ and $\edge(G):=\abs{E(G)}$.
	For $v\in V(G)$, let $N_G(v)$ be the set of neighbors of $v$ in $G$,
	and let $d_G(v):=\abs{N_G(v)}$.
	The \emph{average degree} of $G$ is $\frac{1}{\abs{G}}\sum_{v\in V(G)}d_G(v)=\frac{2\edge(G)}{\abs{G}}$,
	and its \emph{edge density} is $\edge(G)/{\abs{G}\choose2}$.
	Let $\delta(G)$ and $\Delta(G)$, respectively, denote the \emph{minimum degree} and \emph{maximum degree} of $G$.
	Let $\overline{G}$ denote the graph with vertex set $V(G)$ and edge set $\{uv:u,v\in V(G),uv\nin E(G)\}$;
	then $\Delta(\overline{G})=\abs{G}-1-\delta(G)$
	is the maximum number of nonneighbors of a vertex of $G$.
	
	Two disjoint subsets $A,B\subset V(G)$ are \emph{anticomplete} in $G$ if $G$ has no edge between them.
	For $S\subset V(G)$, let $G[S]$ denote the subgraph of $G$ induced by $S$;
	and let $G\setminus S:=G[V(G)\setminus S]$.
	A subset $X\subset V(G)$ is a \emph{cutset} of $G$ if there is a partition $A\cup B=V(G)\setminus X$
	with $A,B$ nonempty and anticomplete.
	For~an~integer $k\ge0$,
	$G$ is \emph{$k$-connected} if $\abs{G}>k$ and $G$ has no cutset of size less than $k$;
	and $G$ is \emph{connected} if $\abs{G}=1$ or it is $1$-connected.
	The \emph{connectivity} of $G$, denoted by $\kappa(G)$, is the minimum size of a cutset of $G$.
	
	For an integer $t\ge1$ and a collection $\mac M=\{M_1,\ldots,M_t\}$ of subgraphs of $G$,
	let $\bigcup_{i=1}^tM_i$ be the subgraph of $G$ on vertex set $V(\mac M):=\bigcup_{i=1}^tV(M_i)$
	and edge set $\bigcup_{i=1}^tE(M_i)$. Given a graph $H$ on vertex set $[t]$, $\mac M$ is
	an {\em $H$ model} in $G$ if $M_1,\ldots,M_t$ are vertex-disjoint and for all distinct $i,j\in[t]$, $G$ has an edge between $V(M_i)$ and $V(M_j)$ whenever $ij\in E(H)$;
	and each $M_i$ is called a \emph{fragment} of $\mac M$.
	Thus $G$ has an $H$ minor if and only if it has an $H$ model.
	For a subset $S\subset V(\mac M)$ of size $t$,
	$\mac M$ is \emph{rooted at $S$} if $\abs{S\cap V(M_i)}=1$ for all $i\in[t]$.
	\section{Proof sketch}
	\label{sec:sketch}
	This section summarizes the proof of \cref{thm:main}, which follows the same route as that of~\cite[Theorem 1.6]{delcourt2021reducing}.
	For integer $m\ge0$, a graph $G$ is \emph{$m$-chromatic-separable} if there are vertex-disjoint subgraphs $G_1,G_2$ of $G$ with $\chi(G_1),\chi(G_2)\ge\chi(G)-m$;
	and $G$ is \emph{$m$-chromatic-inseparable} if there do not exist such $G_1,G_2$.
	Our proof of \cref{thm:main} treats
	the chromatic-inseparable case separately via the following theorem.
	\begin{theorem}
		\label{thm:insep}
		There is an integer $C_{\ref{thm:insep}}>0$ such that the following holds.
		Let $\vep\in(0,\frac{1}{256})$, and let $t\ge\sqrt{\log(1/\vep)}$ be an integer.
		For every graph $G$, let
		\[g_{\ref{thm:insep}}(G,\vep,t)
		:=1+\max_{F\subset G}
		\left\{\frac{\chi(F)}{t}:\abs{F}\le C_{\ref{thm:insep}}t\log^4(1/\vep),\,
		\text{$F$ is $(\vep,t)$-dense-minor-free}\right\}.\]
		If $G$ is $C_{\ref{thm:insep}}t\cdot g(G,\vep,t)$-chromatic-inseparable and $\chi(G)\ge C_{\ref{thm:insep}}t\cdot g_{\ref{thm:insep}}(G,\vep,t)$,
		then $G$ contains an $(\vep,t)$-dense minor.
	\end{theorem}
	The $(\vep,t)$-dense minor will be constructed iteratively in the chromatic-inseparable case and
	recursively in the general case
	(using the chromatic-inseparable case as a blackbox).
	These constructions employ similar setups, which involve three vertex-disjoint parts of the graph in consideration:
	\begin{itemize}
		\item the {\em base}, whose role depends on each case (which we shall explain below);
		
		\item the {\em hub}, which is a small dense subgraph with large connectivity; and 
		
		\item the {\em giant}, which is a highly connected subgraph with large chromatic number.
	\end{itemize}
	
	The existence of the hub will be guaranteed by the following result
	(cf.~\cite[Theorem 2.3]{delcourt2021reducing}).
	\begin{theorem}
		\label{thm:smalldense}
		There is an integer $C_{\ref{thm:smalldense}}>0$ such that,
		for every $\vep\in(0,\frac{1}{3})$ and integers $k\ge t\ge2$,
		every graph with average degree at least $C_{\ref{thm:smalldense}}k$ contains either an $(\vep,t)$-dense minor or a $k$-connected subgraph with at most $C_{\ref{thm:smalldense}}^2t\log^3(1/\vep)$ vertices.
	\end{theorem}
	
	Informally, we would like to link the base and the giant simultaneously to the hub by a collection $\mac P$ of vertex-disjoint paths,
	then group the endpoints of $\mac P$ in the hub suitably and use high connectivity to connect the vertices within each of these groups.
	Let $U$ be the set of endpoints of $\mac P$ in the giant.
	
	In the chromatic-inseparable case, we construct an $(\vep,t)$-dense model after about $\sqrt{\log(1/\vep)}$ iterations.
	In each step, the base consists of the roots of the previously constructed model,
	together with some specified vertices of about $\sqrt{\log(1/\vep)}$ vertex-disjoint small dense subgraphs with similar features as the hub.
	We then do:
	\begin{enumerate}
		\item in each of these small dense subgraphs, use high connectivity to generate a small dense model rooted at the endpoints of $\mac P$;
		
		\item \dd enlarge\ee{} each fragment of the existing model by linking it (via $\mac P$) to an endpoint in $U$
		and several fragments of the small dense models;
		
		\item use the remaining small fragments and vertices in $U$ to form about $\frac{t}{\sqrt{\log(1/\vep)}}$ new fragments and add them to the above enlarged model to get a new dense model;
		
		\item make $U$ the set of new roots;
		
		\item via chromatic-inseparability, \dd attach\ee{} a highly connected subgraph $F$ that contains the giant and has even larger chromatic number to the new model at $U$; then
		
		\item consider $F$ with the new model in the next iteration.
	\end{enumerate}
	
	For the general case, the recursion is based on a ternary tree of depth about $\log\log(1/\vep)$
	whose nodes correspond to the graphs in consideration.
	For each such graph,
	the base consists of the vertices we want to root our $(\vep,a)$-dense model at (for some $a$).
	Let $G'$ be the giant and assume $\abs{U}=a$; we then do:
	\begin{enumerate}
		\item find a subset $V\subset V(G')\setminus U$ of size $2a$
		which has a partition into $a$ pairs each consisting of two neighbors of a vertex of $U$;
		
		\item use \cref{thm:insep} as a black box to create three vertex-disjoint subgraphs of high connectivity and chromatic number in $G'\setminus(U\cup V)$;
		
		\item let these subgraphs be $L_1,L_2,L_3$ and make them the children of our current node in the recursion~tree;
		
		\item appropriately link $V$ to $V(L_1)\cup V(L_2) \cup V(L_3)$ by a collection $\mac Q$ of $2a$ vertex-disjoint paths; then
		
		\item generate an $(\vep,\frac{2}{3}a)$-dense model within each $L_i$ (by recursion) and combine these models (via $\mac Q$) to form an $(\vep,a)$-dense model rooted at $U$ and thus at the designated vertices in the base.
	\end{enumerate}
	
	Except we were cheating slightly here; in fact, combining the three $(\vep,\frac{2}{3}a)$-dense models this way would probably result in an $(\frac{4}{3}\vep,a)$-dense model,
	that is, the nonedge density will increase by a factor~of~$\frac{4}{3}$.
	Therefore it is desirable to keep the depth of the recursion tree about $\log\log(1/\vep)$;
	and after that we should be able to find a small dense model directly within each leaf node of the tree.
	Starting with $a=t$ at the highest node,
	we would expect such a small dense model to have about $\frac{t}{\sqrt{\log(1/\vep)}}$ vertices.
	
	Another technical point that we also cheated in both constructions is that
	the generated small dense models could intersect $\mac P$ (in the chromatic-inseparable case)
	or $\mac Q$ (in the general case) at many vertices.
	Thus it is also desirable to show that given high connectivity,
	it is possible to simultaneously find a dense model and \dd weave\ee{} a given linkage around so that the new linkage only intersects the model at the original endpoints.
	In particular, in the general case, the base should also contain a couple of pairs of vertices to be linked separately so that the construction can be done smoothly.
	
	All of the above issues, together with \cref{thm:smalldense},
	will be resolved as long as we have obtained
	\begin{itemize}
		\item density results for $(\vep,t)$-dense minors in general graphs and unbalanced bipartite graphs; and
		
		\item tools on rooted dense minors and dense wovenness.
	\end{itemize}
	
	These will be addressed in \cref{sec:density,sec:conntool}, respectively;
	and other technical issues will be handled along the way.
	\cref{sec:smalldense} includes a proof of \cref{thm:smalldense}.
	We shall prove \cref{thm:insep} in \cref{sec:insep},
	and prove \cref{thm:main} in \cref{sec:proof}.
	\section{Density results}
	\label{sec:density}
	In this section, we prove two results providing essentially tight bounds on the density of graphs and the asymmetric density of bipartite graphs for $(\vep,t)$-dense minors.
	\begin{theorem}
		\label{thm:denselinearhwg}
		There is an integer $C=C_{\ref{thm:denselinearhwg}}>0$ such that for every $\vep\in(0,\frac{1}{3})$
		and every integer $t\ge2$,
		every graph with average degree at least $Ct\sqrt{\log(1/\vep)}$ contains an $(\vep,t)$-dense minor.
	\end{theorem}
	\begin{theorem}
		\label{thm:avgdegbip}
		There is an integer $C=C_{\ref{thm:avgdegbip}}>0$
		such that the following holds. 
		For every $\vep\in(0,\frac{1}{3})$, every integer $t\ge2$,
		and every bipartite graph $G$ with bipartition $(A,B)$, if
		\[\edge(G)\ge Ct\sqrt{\log(1/\vep)}\sqrt{\abs{A}\abs{B}}
		+t\abs{G},\]
		then $G$ contains an $(\vep,t)$-dense minor.
	\end{theorem}
	\cref{thm:denselinearhwg,thm:avgdegbip}, respectively, are extensions of the Kostochka--Thomason theorem and the asymmetric density theorem developed by Norin and Postle~\cite[Theorem 3.2]{np2021}.
	By random graph considerations mirroring~\cite{MR593989}, the $t\sqrt{\log(1/\vep)}$ term in these theorems can be seen to be tight up to a constant factor.
	The $t\abs{G}$ term in \cref{thm:avgdegbip} is also necessary:
	indeed, let $\abs{A}=a$ for any $a\ge1$ and $\abs{B}=\floor{(1-\sqrt{2\vep})t}$;
	then for every collection of disjoint nonempty connected subsets $\{B_1,\ldots,B_t\}$ of $G$,
	there are at least $t\sqrt{2\vep}$ sets lying entirely within $A$ and thus are pairwise anticomplete, which implies that the corresponding minor has at least ${t\sqrt{2\vep}\choose2}>\vep{t\choose2}$ nonedges for all sufficiently large $t$ (as a function of $\vep$).
	In fact this $t\abs{G}$ term can be improved slightly to $(t-2)\abs{G}$, but we do not need it for our purposes.
	
	We remark that even though \cref{thm:avgdegbip} implies \cref{thm:denselinearhwg},
	the proof of the former makes use of the latter;
	so it is necessary to prove \cref{thm:denselinearhwg} in the first place.
	We need the following lemma,
	which is inspired by the recent idea of Alon, Krivelevich, and Sudakov~\cite{aks2022} for the Kostochka--Thomason bound.
	\begin{lemma}
		\label{lem:build}
		Let $\vep>0$, let $k\ge0$ and $n,r\ge1$ be integers,
		and let $G$ be a graph with $n/6\le \abs{G}\le n$ and 
		\[24^r\left(\frac{\Delta(\overline{G})}{\abs{G}-1}\right)^{r^2}\le\vep.\]
		Then for $A_1,\ldots,A_k\subset V(G)$ with $\abs{A_i}\le \frac{1}{12}\vep^{1/r}n$ for all $i\in[k]$, there exists $S\subset V(G)$ with
		\begin{itemize}
			\item $\abs{S}\le r$;
			
			\item $S\subset A_i$ for at most $\vep k$ indices $i\in[k]$; and
			
			\item at most $\frac{1}{12}\vep^{1/r}n$ vertices in $V(G)\setminus S$ have no neighbors in $S$.
		\end{itemize}
	\end{lemma}
	\begin{proof}
		Let $S$ be a random set of $r$ vertices in $G$ chosen uniformly at random with repetitions.
		For each $i\in[k]$, since $\abs{G}\ge n/6$, the probability that $S\subset A_i$ is at most
		\[\left(\frac{\abs{A_i}}{\abs{G}}\right)^r
		\le \left(\frac{1}{2}\vep^{1/r}\right)^r
		\le\frac{1}{2}\vep,\]
		and so
		the expected number of indices $i\in[k]$ with $S\subset A_i$ is at most $\frac{1}{2}\vep k$.
		Thus by Markov's inequality,
		with probability more than $1/2$,
		$S\subset A_i$ for at most $\vep k$ indices $i\in[k]$.
		
		For each $v\in V(G)$, conditioned on $v\nin S$, the probability that $v$ has no neighbors in $S$ is at most 
		\[\left(\frac{\Delta(\overline{G})}{\abs{G}-1}\right)^r
		\le\frac{1}{24}\vep^{1/r}.\]
		Thus, since $\abs{G}\le n$, the expected number of vertices in $V(G)\setminus S$ with no neighbors in $S$ is at most $\frac{1}{24}\vep^{1/r}n$;
		hence with probability more than $1/2$,
		at most $\frac{1}{12}\vep^{1/r}n$ vertices in $V(G)\setminus S$ have no neighbors in $S$.
		Consequently there is a choice of $S$ with the desired properties.
		This proves~\cref{lem:build}.
	\end{proof}
	We also need the following result of Mader~\cite{MR229550}
	which was mentioned in the Introduction.
	\begin{lemma}
		\label{lem:mader}
		For every integer $d\ge2$, every graph with average degree at least $d-1$ contains a minor with at most $d$ vertices and minimum degree at least $d/2$.
	\end{lemma}
	The next result implicitly appears in~\cite{aks2022}.
	\begin{lemma}
		\label{lem:denseminor}
		For every integer $d\ge2$, every graph with average degree at least $d$ 
		contains a minor with at most $d$ vertices, minimum degree at least $d/3$, and connectivity at least $d/6$.
	\end{lemma}
	\begin{proof}
		Let $G$ be a graph with average degree at least $d$. 
		\cref{lem:mader} gives a minor $H$ of $G$ with $\abs{H}\le d$ and $\delta(H)\ge d/2$.
		We shall find a subgraph $H_0$ of $H$ with $\delta(H_0)\ge d/3$ and $\kappa(H_0)\ge d/6$.
		Indeed, if $\kappa(H)\ge d/6$ then we can let $H_0:=H$;
		so we may assume $\kappa(H)<d/6$.
		Let $X$ be a cutset of $H$ with $\abs{X}<d/6$,
		and let $A\cup B$ be a partition of $V(H)\setminus X$ with $A,B$ nonempty and anticomplete in~$H$.
		By symmetry we may assume $\abs{A}\le\abs{H}/2\le d/2$.
		Since $\delta(H)\ge d/2$ and $\abs{X}<d/6$,
		each vertex in $A$ has more than $d/3$ neighbors in $A$.
		Thus, since $\abs{A}\le d/2$, every two vertices in $A$ have more than $2d/3-d/2=d/6$ common neighbors in $A$;
		so $\kappa(H[A])>d/6$ and we let $H_0:=H[A]$.
		This proves~\cref{lem:denseminor}.
	\end{proof}
	We are now ready to prove~\cref{thm:denselinearhwg}.
	\begin{proof}
		[Proof of~\cref{thm:denselinearhwg}]
		Let $C:=3360$.
		Let $d:=\ceil{Ct\sqrt{\log(1/\vep)}}$ and $r:=\ceil{20\sqrt{\log(1/\vep)}}$;
		then $d\ge 84t\cdot 40\sqrt{\log(1/\vep)}
		\ge 84rt$.
		Let $G$ be a graph with average degree at least $d$.
		By~\cref{lem:denseminor},
		$G$ has a minor $H$ with $\abs{H}\le d$,
		$\delta(H)\ge d/3$,
		and $\kappa(H)\ge d/6$.
		We shall build an $(\vep,t)$-dense minor in $H$ by the following claim.
		\begin{claim}
			\label{claim:minor}
			There exist disjoint nonempty subsets $B_1,\ldots,B_t$ of $ V(H)$
			such that for each $i\in[t]$,
			$H[B_i]$ is connected
			and $B_i$ is anticomplete to at most $\vep(i-1)$ sets among $B_1,\ldots,B_{i-1}$ in $H$.
		\end{claim}
		\begin{subproof}
			Put $\gamma:=\frac{1}{12}\vep^{1/r}$.
			Let $k$ be an integer with $0\le k<t$; and assume we have constructed
			disjoint subsets $B_1,\ldots,B_k$ of $ V(H)$ such that 
			for each $i\in[k]$,
			\begin{itemize}
				\item $\abs{B_i}\le 14r$;
				
				\item $H[B_i]$ is connected;
				
				\item $B_i$ is anticomplete to at most $\vep(i-1)$ sets among $B_1,\ldots,B_{i-1}$ in $H$; and
				
				\item at most $\gamma d$ vertices in $V(H)\setminus(B_1\cup\cdots\cup B_{i-1})$ have no neighbors in $B_i$.
			\end{itemize} 
			
			Let $F:=H\setminus(B_1\cup\cdots\cup B_k)$.
			Since $\abs{B_1}+\cdots+\abs{B_k}
			\le 14rk<14rt\le d/6$, we have
			$\delta(F)>\delta(H)-d/6\ge d/6$ and $\kappa(F)> \kappa(H)-d/6\ge 0$;
			thus $d/6\le\abs{F}\le d$, $F$ is connected, and by the choice of $r$,
			\[24^r\left(\frac{\Delta(\overline{F})}{\abs{F}-1}\right)^{r^2}
			=24^r\left(1-\frac{\delta(F)}{\abs{F}-1}\right)^{r^2}
			\le 24^r\left(\frac{5}{6}\right)^{r^2}
			\le\left(\frac{5}{6}\right)^{r^2/10}
			\le \vep.\]
			
			For $i\in[k]$,
			let $A_i$ be the set of vertices in $V(F)$ with no neighbors in $B_i$;
			then $\abs{A_i}\le \gamma d$.
			By~\cref{lem:build} applied to $F$ and $A_1,\ldots,A_k$, there exists $S\subset V(F)$ such that
			\begin{itemize}
				\item $\abs{S}\le r$;
				
				\item $S\subset A_i$ for at most $\vep k$ indices $i\in[k]$,
				which means that $S$ is anticomplete to at most $\vep k$ sets among $B_1,\ldots,B_k$ in $H$; and
				
				\item at most $\gamma d$ vertices in $V(F)\setminus S$ have no neighbors in $S$.
			\end{itemize}
			
			Since $F$ is connected, every two vertices in $S$ is joined by some induced path in $V(F)$;
			and since $\delta(F)\ge d/6$,
			such a path have length at most $14$,
			because every two vertices of distance at least three on it should have disjoint neighborhoods in $F$.
			So we can add at most $13\abs{S}$ vertices from $V(F)\setminus S$ to $S$ to get a set $B_{k+1}\subset V(F)$ with $\abs{B_{k+1}}\le 14\abs{S}\le 14r$ such that $F[B_{k+1}]=H[B_{k+1}]$ is connected.
			Hence,
			\begin{itemize}
				\item $\abs{B_{k+1}}\le 14r$;
				
				\item $H[B_{k+1}]$ is connected;
				
				\item $B_{k+1}$ is anticomplete to at most $\vep k$ sets among $B_1,\ldots,B_k$ in $H$; and
				
				\item at most $\gamma d$ vertices in $V(F)\setminus B_{k+1}=V(H)\setminus (B_1\cup\cdots\cup B_{k+1})$ have no neighbors in $B_{k+1}$.
			\end{itemize}
			
			Iterating this for $k=0,1,\ldots,t-1$ in turn proves~\cref{claim:minor}.
		\end{subproof}
		Now, contracting each of $B_1,\ldots,B_t$ in $H$
		gives a minor of $H$ (and so of $G$)
		with $t$ vertices and at least $ (1-\vep)(1+2+\cdots+(t-1))=(1-\vep){t\choose2}$ edges.
		This proves~\cref{thm:denselinearhwg}.
	\end{proof}
	\cref{thm:denselinearhwg} gives an asymptotically optimal criterion for general graphs containing a dense minor.
	Here is a related result (cf.~\cite[Lemma 3.1]{np2021}) with a similar proof.
	\begin{lemma}
		\label{lem:avgdegbip1}
		Let $\vep\in(0,1)$, and let $n,t\ge2$ be integers with $n\ge12t$.
		Let $G$ be a graph with $\abs{G}=n$.
		Put $r:=\floor{\frac{n}{12t}}$ and $q:=\edge(\overline{G})/{n\choose2}$.
		If $24^r(12q)^{r^2}\le\vep$,
		then $G$ contains an $(\vep,t)$-dense minor.
	\end{lemma}
	\begin{proof}
		Since $(12q)^{r^2}<\vep<1$, we have $q<\frac{1}{12}$.
		Since $\edge(\overline{G})\le\frac{1}{2}qn^2$,
		there exists $S\subset V(G)$ with $\abs{S}=\floor{\frac{1}{3}n}$ such that every vertex of $S$ has at most $2qn$ nonneighbors in $G$;
		in particular $\Delta(\overline{G}[S])\le 2qn$.
		\begin{claim}
			\label{claim:build}
			There exist disjoint nonempty subsets $S_1,\ldots,S_t$ of $S$ such that for each $i\in[t]$, $S_i$ is anticomplete to at most $\vep(i-1)$ sets among $S_1,\ldots,S_{i-1}$ in $G[S]$.
		\end{claim}
		\begin{subproof}
			Put $\gamma:=\frac{1}{12}\vep^{1/r}$.
			Let $k$ be an integer with $0\le k<t$; and assume we have constructed disjoint nonempty subsets $S_1,\ldots,S_k$ of $S$ such that for each $i\in[k]$,
			\begin{itemize}
				\item $\abs{S_i}\le r$;
				
				\item $S_i$ is anticomplete to at most $\vep(i-1)$ sets among $S_1,\ldots,S_{i-1}$; and
				
				\item at most $\gamma n$ vertices in $S\setminus(S_1\cup\cdots\cup S_{i-1})$ have no neighbors in $S_i$.
			\end{itemize}
			Let $F:=G[S]\setminus(S_1\cup\cdots\cup S_k)$; then since
			$\abs{S_1}+\cdots+\abs{S_k}
			\le rk<rt\le \frac{1}{12}n$
			and $n\ge 12t\ge 24$, we have
			\[\abs{F}-1\ge\abs{S}-\frac{n}{12}-1
			\ge\frac{n}{3}-1-\frac{n}{12}-1
			=\frac{n}{4}-2
			\ge\frac{n}{6}.\]
			Hence, since $\Delta(\overline{F})\le\Delta(\overline{G}[S])\le 2qn$,
			\[24^r\left(\frac{\Delta(\overline{F})}{\abs{F}-1}\right)^{r^2}
			\le24^r(12q)^{r^2}\le\vep.\]
			
			For each $i\in[k]$, let $A_i$ be the set of vertices in $V(F)$ with no neighbors in $S_i$; then $\abs{A_i}\le\gamma n$.
			By~\cref{lem:build} applied to $F$ with $A_1,\ldots,A_t$, there exists $S_{k+1}\subset V(F)$ with
			\begin{itemize}
				\item $\abs{S_{k+1}}\le r$;
				
				\item $S_{k+1}\subset A_i$ for at most $\vep k$ indices $i\in[k]$,
				which means that $S_{k+1}$ is anticomplete to at most $\vep k$ sets among $S_1,\ldots,S_k$; and
				
				\item at most $\gamma n$ vertices in $V(F)\setminus S_{k+1}=S\setminus(S_1\cup\cdots\cup S_{k+1})$
				have no neighbors~in~$S_{k+1}$.
			\end{itemize}
			
			Iterating this for $k=0,1,\ldots,t-1$ in turn proves~\cref{claim:build}.
		\end{subproof}
		Now, let $S_1,\ldots,S_t$ be obtained from \cref{claim:build}.
		Since every vertex in $S$ has at most $2qn$ nonneighbors in $G$,
		every two nonadjacent vertices in $S$ have at least $(1-4q)n$ common neighbors in $G$,
		and so at least $(1-4q)n-\abs{S}\ge \frac{2}{3}n-\abs{S}\ge\abs{S}$ common neighbors in $V(G)\setminus S$;
		note that $q<\frac{1}{12}$.
		Thus we can add vertices from $V(G)\setminus S$ to $S_1,\ldots,S_t$ to get disjoint nonempty subsets $B_1,\ldots,B_t$ of $V(G)$ such that $S_i\subset B_i$ and $G[B_i]$ is connected for all $i\in[t]$.
		Therefore $\{G[B_1],\ldots,G[B_t]\}$ is an $H$ model in $G$ for some graph $H$ with
		$\edge(\overline{H})\le\vep(1+2+\cdots+(t-1))=\vep{t\choose2}$.
		This proves~\cref{lem:avgdegbip1}.
	\end{proof}
	We now give a proof of \cref{thm:avgdegbip},
	which is adapted from the proof of~\cite[Theorem 3.2]{np2021}.
	\begin{proof}
		[Proof of~\cref{thm:avgdegbip}]
		Let $C:=100C_{\ref{thm:denselinearhwg}}$.
		Let 
		$d:=t\sqrt{\log(1/\vep)}$,
		and let $G$ be a counterexample to \cref{thm:avgdegbip} with bipartition $(A,B)$ and with $\abs{G}+\edge(G)$ minimal;
		then $G$ has no $(\vep,t)$-dense minor while
		\[\edge(G)\ge Cd \sqrt{\abs{A}\abs{B}}+t\abs{G}.\]
		
		By symmetry, we may assume that $\abs{A}\ge\abs{B}$.
		Put $p:=\sqrt{\abs{A}/\abs{B}}\ge1$.
		\begin{claim}
			\label{claim:avgdegbip}
			$p\ge5$, and there exists $u_0\in A$ with $d_G(u_0)<C_{\ref{thm:denselinearhwg}}d$.
		\end{claim}
		\begin{subproof}
			
			For every $u\in A$,
			the minimality of $G$ implies that
			\[\edge(G\setminus u)< Cd
			\sqrt{(\abs{A}-1)\abs{B}}+ t(\abs{G}-1),\]
			and so
			\begin{equation*}
				d_G(u)=\edge(G)-\edge(G\setminus u)
				>Cd\left(\sqrt{\abs{A}}-\sqrt{\abs{A}-1}\right)\sqrt{\abs{B}}
				+t
				>\frac{1}{2}Cdp^{-1}+t.
			\end{equation*}
			Similarly, we have
			$d_G(v)>\frac{C}{2}dp+t$ for all $v\in B$.
			So, since $t\ge2$, $q\ge1$, and $C\ge10C_{\ref{thm:denselinearhwg}}$,
			we see that $d_G(v)>C_{\ref{thm:denselinearhwg}}d$ for all $v\in B$;
			thus by~\cref{thm:denselinearhwg}, since $G$ has no $(\vep,t)$-dense minor,
			there exists $u_0\in A$ with $d_G(u_0)<C_{\ref{thm:denselinearhwg}}d$.
			Thus $C_{\ref{thm:denselinearhwg}}d
			>\frac{1}{2}Cdp^{-1}$ and so $p>C/(2C_{\ref{thm:denselinearhwg}})\ge5$.
			This proves \cref{claim:avgdegbip}.
		\end{subproof}
		Now, let $v_1,v_2$ be distinct neighbors of $u_0$,
		and let $G'$ be the graph obtained from $G$ by deleting $u_0$ and identifying $v_1$ and $v_2$;
		then $G'$ is a minor of $G$.
		The minimality of $G$ implies that 
		\[\edge(G')< Cd\sqrt{(\abs{A}-1)(\abs{B}-1)}+ t(\abs{G}-2)
		< Cd\sqrt{\abs{A}(\abs{B}-1)}+ t\abs{G}\]
		and so, by \cref{claim:avgdegbip}, the number of common neighbors of $v_1,v_2$ in $G$ is
		\[\edge(G)-\edge(G')-d_G(u_0)
		>\frac{1}{2}Cdp
		-C_{\ref{thm:denselinearhwg}}d
		\ge \frac{1}{4}Cdp. \]
		Hence the number of common neighbors of every two neighbors of $u_0$ is at least $\ceil{\frac{1}{4}Cdp}=:s$.
		
		Now, put $n:=\ceil{\frac{1}{2}Cdp^{-1}+t}\ge t$,
		and let $S$ be an arbitrary subset of $N_G(u_0)$ of size $n$.
		Let $H$ be a random minor of $G$ obtained as follows:
		every vertex $u\in A\setminus\{u_0\}$ with a neighbor in $S$ picks one of its neighbors in $S$, say $\varphi(u)$, independently and uniformly at random;
		then we contract the edges $\varphi(u)u$.
		Given distinct $v_1,v_2\in S$,
		we see that they are nonadjacent in $H$ if and only if none of their common neighbors picked one of them,
		which occurs with probability at most
		\[\left(1-\frac{2}{n}\right)^s\le e^{-2s/n}=:q.\]
		Hence, there is a choice of $H$ for which the number of its nonedges is at most ${n\choose2}q$.
		If $q<\vep$ then $G$ would contain an $(\vep,t)$-dense minor by averaging (as $n\ge t$), a contradiction.
		We may thus assume $q\ge\vep$,
		and so $\vep^{-1}\ge q^{-1}=e^{2s/n}$
		which yields $n\log(1/\vep)\ge 2s$.
		Since $n\ge\frac{1}{2}Cdp^{-1}$,
		$s\ge \frac{1}{4}Cdp$,
		and $C\ge 100$, we have that
		$n^2\log(1/\vep)\ge 2ns\ge \frac{1}{4}C^2d^2
		\ge 500t^2\log(1/\vep)$,
		and so $n\ge 60t$.
		Put $r:=\floor{\frac{n}{12t}}$;
		then $r\ge\max(5,\frac{n}{15t})$.
		Since $s\ge np\ge 5n$, we have $24q=24e^{-2s/n}\le e^{-s/n}$.
		Also,
		\[r^2\frac{s}{n}
		\ge\left(\frac{n}{15t}\right)^2\frac{s}{n}
		=\frac{ns}{225t^2}
		\ge \log(1/\vep),\]
		and so
		$24^r(12q)^{r^2}
		\le (24q)^{r^2}
		\le e^{-r^2s/n}\le \vep$. 
		By \cref{lem:avgdegbip1}, $H$ (and so $G$) would contain an $(\vep,t)$-dense minor, a contradiction.
		This proves~\cref{thm:avgdegbip}.
	\end{proof}
	\section{Small dense subgraphs}
	\label{sec:smalldense}
	With \cref{thm:denselinearhwg,thm:avgdegbip} in hand,
	it is not hard to adapt the proof of \cite[Theorem 2.3]{delcourt2021reducing}
	to prove \cref{thm:smalldense};
	still, we give a proof for completeness.
	We need the following lemma from~\cite[Lemma 4.2]{delcourt2021reducing}.
	\begin{lemma}
		\label{lem:increment}
		Let $r\ge3$ be an integer, and let $\delta>0$.
		Let $G$ be a graph and let $S\subset V(G)$ satisfy
		\[(r-2)\cdot\edge_G(S)
		>(r-1)\delta\abs{S}+\partial_G(S).\]
		Then there exists a nonempty subset $S'$ of $ S$ such that $d_{G[S']}(v)\ge\max(\delta,\frac{1}{r}d_G(v))$ for all $v\in S'$.
	\end{lemma}
	We also need a classical result of Mader~\cite{MR229550} (see also~\cite[Theorem 1.4.3]{MR3822066}).
	\begin{lemma}
		\label{lem:avgkappa}
		For integer $k\ge1$,
		every graph with average degree at least $4k$ has a $k$-connected subgraph.
	\end{lemma}
	We are now ready to prove \cref{thm:smalldense}, which we restate here for the convenience of the readers.
	\begin{theorem}
		\label{thm:sdense}
		There is an integer $C>0$ such that,
		for every $\vep\in(0,\frac{1}{3})$ and integers $k\ge t\ge2$,
		every graph with average degree at least $Ck$ contains either an $(\vep,t)$-dense minor or a $k$-connected subgraph with at most $C^2t\log^3(1/\vep)$ vertices.
	\end{theorem}
	\begin{proof}
		Let $C:=400\max(C_{\ref{thm:denselinearhwg}},C_{\ref{thm:avgdegbip}})$.
		Put $\Gamma:=\log(1/\vep)$;
		then by \cref{thm:denselinearhwg},
		we may assume $k\le \Gamma^{1/2}t$.
		Let $G$ be a graph with average degree $d\ge Ck$;
		we may assume $Ck\le d\le 2Ck$.
		Let $m\ge0$ be maximal such that there exist vertex-disjoint subgraphs $F_1,\ldots,F_m$ of $G$ satisfying
		\begin{itemize}
			\item $F_i$ is connected and $\abs{F_i}=\ceil{10\Gamma^2(t/k)^2}$ for all $i\in[m]$; and
			
			\item the graph $G'$ obtained from $G$ by contracting each $F_i$ into a vertex $x_i$ satisfies
			\[\edge(G)-\edge(G')\le\frac{d}{20}(\abs{G}-\abs{G'}).\]
		\end{itemize}
		Such a collection $F_1,\ldots,F_m$ always exists because the empty collection satisfies the two bullets.
		Let
		\begin{align*}
			X&:=\{x_i:i\in[m]\},\\
			Y&:=\{v\in V(G')\setminus X:\abs{N_{G'}(v)\cap X}\ge {\textstyle \frac{1}{30}d}\},\\
			Z&:=\{v\in V(G')\setminus (X\cup Y):d_{G'}(v)\ge20d\Gamma\}.
		\end{align*}
		Note that $G'$ is $(\vep,t)$-dense-minor-free since $G$ is, and that $\abs{X}=m\le \frac{1}{10}\Gamma^{-2}(k/t)^2\abs{G}
		\le\frac{1}{10}\Gamma^{-1}\abs{G}$.
		In what follows, for disjoint $S,T\subset V(G')$,
		let $\edge_{G'}(S,T)$ be the number of edges of $G'$ between $S$ and $T$.
		\begin{claim}
			\label{claim:smalldense1}
			$\abs{Y}\le \Gamma\cdot(t/k)^2\abs{X}$,
			and so $\abs{Y}\le\frac{1}{10}\Gamma^{-1}\abs{G}$.
		\end{claim}
		\begin{subproof}
			Since $k\le \Gamma^{-1/2}t$, we may assume $\abs{Y}\ge\abs{X}$.
			By~\cref{lem:avgdegbip1} and since $G'$ is $(\vep,t)$-dense-minor-free,
			\[\edge_{G'}(X,Y)\le C_{\ref{thm:avgdegbip}}\Gamma^{1/2}t\sqrt{\abs{X}\abs{Y}}+t(\abs{X}+\abs{Y})
			\le C_{\ref{thm:avgdegbip}}\Gamma^{1/2}t\sqrt{\abs{X}\abs{Y}}+2t\abs{Y}.\]
			Now, since $\abs{N_{G'}(v)\cap X}\ge d/30$ for all $v\in Y$, since $C\ge 60C_{\ref{thm:avgdegbip}}$,
			and since $d\ge C_{\ref{thm:avgdegbip}}k\ge 2t$,
			we see that 
			\[\edge_{G'}(X,Y)\ge \frac{d}{30}\abs{Y}\ge 2C_{\ref{thm:avgdegbip}}k\abs{Y}
			\ge C_{\ref{thm:avgdegbip}}k\abs{Y}+2t\abs{Y}.\]
			It follows that $k\abs{Y}\le \Gamma^{1/2}t\sqrt{\abs{X}\abs{Y}}$
			and so $\abs{Y}\le \Gamma\cdot(t/k)^2\abs{X}$.
			This proves~\cref{claim:smalldense1}.
		\end{subproof}
		Now, let $T:=X\cup Y\cup Z$ and $S:=V(G')\setminus T$.
		\begin{claim}
			\label{claim:smalldense2}
			$\edge(G'[S])>\frac{2d}{5}\abs{S}+\edge_{G'}(S,T)$.
		\end{claim}
		\begin{subproof}
			Since $d_{G'}(v)\ge20d\Gamma$ for all $v\in Z$, we have $\abs{Z}\le
			2\cdot\edge(G)/(20d\Gamma)
			=\frac{1}{10}\Gamma^{-1}\abs{G}$.
			Thus,~\cref{claim:smalldense1} implies that $\abs{T}=\abs{X}+\abs{Y}+\abs{Z}\le \frac{3}{10}\Gamma^{-1}\abs{G}$.
			Since $\abs{S}\le\abs{G}$,
			by~\cref{thm:avgdegbip}, we have
			\[e_{G'}(S,T)
			\le C_{\ref{thm:avgdegbip}}\Gamma^{1/2}t\sqrt{\abs{S}\abs{T}}+t(\abs{S}+\abs{T})
			\le C_{\ref{thm:avgdegbip}}\Gamma^{1/2}t\sqrt{\abs{G}\cdot\frac{3}{10}\Gamma^{-1}\abs{G}}+2t\abs{G}
			\le C_{\ref{thm:avgdegbip}}t\abs{G}.\]
			Since $G'[T]$ is $(\vep,t)$-dense-minor-free, we also have 
			\[\edge(G'[T])\le C_{\ref{thm:denselinearhwg}}\Gamma^{1/2}t\abs{T}
			\le C_{\ref{thm:denselinearhwg}}\Gamma^{1/2}t\cdot \Gamma^{-1}\abs{G}
			\le C_{\ref{thm:denselinearhwg}}t\abs{G}. \]
			Hence, because $\edge(G')>\edge(G)-\frac{1}{20}d\abs{G}
			=\frac{9}{20}d\abs{G}$, $k\ge t$,
			and $C_{\ref{thm:denselinearhwg}}k,C_{\ref{thm:avgdegbip}}k\le \frac{1}{80}d$, 
			it follows that
			\begin{align*}
				\edge(G'[S])-\edge_{G'}(S,T)
				&=\edge(G')-2\cdot\edge_{G'}(S,T)-\edge_{G'}(T)\\
				&> \frac{9}{20}d\abs{G}-2C_{\ref{thm:avgdegbip}}k\abs{G}-C_{\ref{thm:denselinearhwg}}k\abs{G}
				\ge \frac{9}{20}d\abs{G}-\frac{1}{40}d\abs{G}-\frac{1}{80}d\abs{G}
				>\frac{2}{5}d\abs{G}.
				\qedhere
			\end{align*}
		\end{subproof}
		By~\cref{claim:smalldense2} and~\cref{lem:increment} applied to $G'[S]$ with $r=3$ and $\delta=\frac{1}{5}d$,
		there exists $S'\subset S$ such that $d_{G'[S']}(v)\ge \max(\frac{1}{5}d,\frac{1}{3}d_{G'}(v))$ for all $v\in S'$.
		Let $F$ be a nonempty subgraph of $G'[S']$ such that
		\begin{itemize}
			\item $F$ is connected and $\abs{F}\le\ceil{10\Gamma^2(t/k)^2}$;
			
			\item the graph $G''$ obtained from $G'$ by contracting $F$ into a vertex $x$ satisfies
			\[\edge(G')-\edge(G'')\le\frac{d}{20}(\abs{F}-1);\]
			
			\item subject to the above two bullets, $\abs{F}$ is maximal.
		\end{itemize} 
		Such an $F$ always exists since every one-vertex subgraph of $G'[S']$ satisfies the first two bullets.
		Let $R:=N_{G''}(x)\setminus X$, and let $R':=N_{G''}(x)\cap S'\subset R\setminus Y$.
		\begin{claim}
			\label{claim:smalldense3}
			$\abs{R'}\ge\frac{1}{6}\abs{R}$.
		\end{claim}
		\begin{subproof}
			Because $d_{G'[S']}(v)\ge \frac{1}{5}d$ for all $v\in V(F)\subset S'$,
			$\sum_{v\in V(F)}d_{G'[S']}(v)\ge \frac{1}{5}d\abs{F}\ge 4(\edge(G')-\edge(G''))$.
			It follows that
			\[\abs{R'}\ge\sum_{v\in V(F)}d_{G'[S']}(v)-2(\edge(G')-\edge(G''))
			\ge\frac{1}{2}\sum_{v\in V(F)}d_{G'[S']}(v).\]
			Now, since $d_{G'[S']}(v)\ge \frac{1}{3}d_{G'}(v)$ for all $v\in V(F)\subset S'$ and since $R\subset \bigcup_{v\in V(F)}N_{G'}(v)$, we obtain
			\[\abs{R'}\ge\frac{1}{2}\sum_{v\in V(F)}d_{G'[S']}(v)
			\ge \frac{1}{6}\sum_{v\in V(F)}d_{G'}(v)
			\ge\frac{1}{6}\abs{R}.\qedhere\]
		\end{subproof}
		Now, let $v\in R'$. 
		The maximality of $m$ yields $\abs{F}<\ceil{10\Gamma^2(t/k)^2}$.
		So, by the maximality of $\abs{F}$, the graph obtained from $G'$ by contracting $G'[V(F)\cup\{v\}]$ has fewer than
		$\edge(G')-\frac{1}{20}d\abs{F}\le \edge(G'')-\frac{1}{20}d$ edges, which yields
		$\abs{N_{G''}(x)\cap N_{G''}(v)}>\frac{1}{20}d-1$.
		Since $v\in R'\subset R\setminus Y$,
		$\abs{N_{G''}(v)\cap X}\le \frac{1}{30}d$; hence
		\[d_{G[R]}(v)\ge \abs{N_{G''}(v)\cap N_{G''}(x)}-\abs{N_{G''}(v)\cap X}
		>\frac{d}{20}-1-\frac{d}{30}
		=\frac{d}{60}-1\ge\frac{d}{120}.\]
		Therefore, by~\cref{claim:smalldense3},
		\[2\cdot\edge(G[R])
		\ge\sum_{v\in R'}d_{G[R]}(v)
		\ge \frac{1}{120}d\abs{R'}
		\ge \frac{1}{720}d\abs{R}\]
		so $G[R]$ has average degree at least $\frac{1}{720}d\ge 4k$.
		By \cref{lem:avgkappa}, $G[R]$ contains a $k$-connected subgraph $H$.
		Finally, since $d_{G'}(v)\le 20\Gamma d$ for all $v\in V(F)\subset S$, we have
		\[\abs{R}\le\sum_{v\in V(F)}d_{G'}(v)
		\le 20\Gamma d\abs{F}
		\le400C\Gamma^3k(t/k)^2
		\le C^2\Gamma^3t\]
		where the third inequality holds as $d\le 2Ck$ and $\abs{F}\le 10\Gamma^2(t/k)^2$,
		and the last inequality holds as $k\ge t$.
		Hence $\abs{H}\le \abs{R}\le  C^2\Gamma^3t$.
		This proves~\cref{thm:sdense}.
	\end{proof}
	The following handy consequence of \cref{thm:smalldense} (cf.~\cite[Corollary 6.1]{delcourt2021reducing}) will be needed later on.
	\begin{corollary}
		\label{cor:smalldense}
		Let $k\ge t\ge2$ be integers.
		Let $\vep>0$, and let $\Gamma:=\log(1/\vep)$.
		For every graph $G$, put
		\[g(G,\vep,t)=g_{\ref{cor:smalldense}}(G,\vep,t)
		:=1+\max_{F\subset G}\left\{\frac{\chi(F)}{t}:
		\abs{F}\le C_{\ref{thm:smalldense}}^2\Gamma^4t,\,
		\text{$F$ is $(\vep,t)$-dense-minor-free}\right\}.\]
		If $\chi(G)\ge3C_{\ref{thm:smalldense}}k\cdot g(G,\vep,t)$, then $G$ contains at least $\Gamma$ vertex-disjoint $k$-connected subgraphs each with at most $C_{\ref{thm:smalldense}}^2\Gamma^3t$ vertices.
	\end{corollary}
	\begin{proof}
		Let $m\ge0$ be maximal such that there is a collection $F_1,\ldots,F_m$ of vertex-disjoint $k$-connected subgraphs of $G$ each with at most $C_{\ref{thm:smalldense}}^2\Gamma^3t$ vertices.
		We may assume $m<\Gamma$.
		Let $F$ be the subgraph of $G$ induced by $\bigcup_{i\in[m]}V(F_i)$;
		then $\abs{F}\le C_{\ref{thm:smalldense}}^2\Gamma^3t\cdot m\le C_{\ref{thm:smalldense}}^2\Gamma^4t$
		and so $\chi(F)\le t\cdot g(G,\vep,t)$. Let $F':=G\setminus F$;
		then by~\cref{thm:smalldense} and the maximality of $m$,
		every subgraph of $F'$ has average degree at most $C_{\ref{thm:smalldense}}k$,
		and so $F'$ has degeneracy at most $C_{\ref{thm:smalldense}}k$
		which yields $\chi(F')\le C_{\ref{thm:smalldense}}k+1$.
		It follows that
		\[\chi(G)\le \chi(F)+\chi(F')
		\le t\cdot g(G,\vep,t)+C_{\ref{thm:smalldense}}k+1
		<3C_{\ref{thm:smalldense}}k\cdot g(G,\vep,t),\]
		a contradiction.
		This proves~\cref{cor:smalldense}.
	\end{proof}
	\section{Connectivity tools}
	\label{sec:conntool}
	\subsection{Basic tools}
	In this subsection, we recall several definitions and known results on connectivity.
	For a graph $G$ and subsets $S,T\subset V(G)$,
	a \emph{path between $S$ and $T$}
	is a path with one endpoint in $S$, the other endpoint in $T$,
	and no interval vertex in $S\cup T$;
	thus an edge between $S$ and $T$ is a path of length one between $S$ and $T$.
	A \emph{separation} of $G$ is a pair $(A,B)$ of subsets of $V(G)$ with $A\cup B=V(G)$ and $A\setminus B,B\setminus A$ anticomplete in $G$;
	and the \emph{order} of $(A,B)$ is $\abs{A\cap B}$.
	Thus $A\cap B$ is a cutset of $G$ if and only if $A\setminus B,B\setminus A\ne\emptyset$.
	We start with Menger's theorem~\cite{menger1927} and its \dd redundancy\ee{} version developed by Delcourt and Postle~\cite[Lemma 5.3]{delcourt2021reducing}.
	\begin{theorem}
		\label{thm:menger}
		Let $G$ be a graph, and let $S,T\subset V(G)$.
		Then for every integer $k\ge0$, either
		\begin{itemize}
			\item there are $k$ vertex-disjoint paths of $G$ between $S$ and $T$; or
			
			\item there is a separation $(A,B)$ of $G$ of order less than $k$ with $S\subset A$ and $T\subset B$.
		\end{itemize}
	\end{theorem}
	\begin{lemma}
		\label{lem:redundant}
		Let $G$ be a graph,
		and let $S_1,S_2,T$ be disjoint subsets of $V(G)$ such that for each $i\in\{1,2\}$,
		there are $2\abs{S_i}$ paths between $S_i$ and $T$ in $G\setminus T_{3-i}$ which pairwise share no vertex in $V(G)\setminus S_i$, such that each vertex in $S_i$ is the endpoint of exactly two of them.
		Then there are $\abs{S_1}+\abs{S_2}$ vertex-disjoint paths between $S_1\cup S_2$ and $T$ in $G$.
	\end{lemma}
	For an integer $k\ge0$, a graph $G$, and subsets $\{s_1,\ldots,s_k\}$ and $\{t_1,\ldots,t_k\}$ of $V(G)$
	such that $s_i\ne t_j$ for all distinct $i,j\in[k]$,
	a \emph{linkage in $G$ linking} $\{(s_i,t_i):i\in[k]\}$ is a collection $\mac P=\{P_1,\ldots,P_k\}$ of vertex-disjoint paths in $G$ such that $P_i$ has endpoints $s_i$ and $t_i$ for all $i\in[k]$;
	and $G$ is \emph{$k$-linked} if for every such choice of $\{s_1,\ldots,s_k\}$ and $\{t_1,\ldots,t_k\}$,
	there is a linkage in $G$ linking $\{(s_i,t_i):i\in[k]\}$.
	Bollob\'{a}s and Thomason~\cite{MR1417341} were the first to prove that linear connectivity is sufficient to force high linkedness;
	and the current record on this problem is due to Thomas and Wollan~\cite{MR2116174} who proved that $10k$-connectivity yields $k$-linkedness.
	\begin{theorem}
		\label{thm:linked}
		For every integer $k\ge1$,
		every $10k$-connected graph is $k$-linked.
	\end{theorem}
	For integers $k\ge\ell\ge1$,
	$G$ is \emph{$(k,\ell)$-knit} if for every integer $m$ with $\ell\le m\le k$ and every $S\subset V(G)$ with $\abs{S}=k$ and every partition of $S$ into nonempty subsets $S_1,\ldots,S_m$,
	there are vertex-disjoint connected subgraphs $G_1,\ldots,G_m$ of $G$
	such that $S_i\subset V(G_i)$ for all $i\in[m]$.
	Thus $G$ is $k$-linked if it is $(2k,k)$-knit.
	Note that connectivity linear in $k-\ell$ is insufficient to force the $(k,\ell)$-knit property:
	indeed, given integers $C,a\ge1$,
	let $k\ge(C+1)a+1$, and let $G$ be $K_k$ with all edges in a clique of size $a+1$ removed;
	then $\kappa(G)\ge k-a-1\ge Ca$ but $G$ is not $(k,k-a)$-knit.
	Bollob\'{a}s and Thomason~\cite{MR1417341}
	observed that their method for proving linkedness given linear connectivity
	can be adapted to show the existence of some constant $C>0$ such that every $Ck$-connected graph is $(k,\ell)$-knit.
	Here we give a quick proof, using \cref{thm:linked}, that $11k$-connectivity is enough.
	\begin{theorem}
		\label{thm:knit}
		For all integers $k\ge\ell\ge1$,
		every $11k$-connected graph is $(k,\ell)$-knit.
	\end{theorem}
	\begin{proof}
		Let $G$ be a graph with $\kappa(G)\ge 11k$.
		Let $m$ be an integer with $\ell\le m\le k$,
		let $S\subset V(G)$ with $\abs{S}=k$
		and let $S_1\cup\cdots\cup S_m$ be a partition of $S$ into nonempty subsets.
		For every $i\in[m]$,
		let $k_i:=\abs{S_i}$ and $S_i=\{u_{i}^j:j\in[k_i]\}$;
		and we may assume $k_1\ge\ldots\ge k_m\ge1$.
		Let $n\ge0$ be maximal such that $k_n\ge2$.
		Then it suffices to construct vertex-disjoint connected subgraphs $G_1,\ldots,G_n$ of $G\setminus(\bigcup_{i=n+1}^mS_i)$ with $S_i\subset V(G_i)$ for all $i\in[n]$.
		Since $\delta(G)\ge\kappa(G)\ge11k\ge3\abs{S}$, there exists a set of $2n$ vertices $\bigcup_{i=1}^n\{v_{i1}^j,v_{i2}^j:j\in[k_i]\}\subset V(G)\setminus S$
		such that for every $i\in[n]$, $u_i^j$ is adjacent to $v_{i1}^j,v_{i2}^j$ for all $j\in[k_i]$.
		Since $\kappa(G\setminus S)\ge 11k-k= 10k$,
		$G\setminus S$ is $k$-linked by \cref{thm:linked};
		thus there is a linkage in $G\setminus S$ linking 
		$\bigcup_{i=1}^n\{(v_{i2}^j,v_{i1}^{j+1}):j\in[k_i-1]\}$.
		This linkage gives rise to the desired vertex-disjoint connected subgraphs $G_1,\ldots,G_n$ of $G$,
		proving \cref{thm:knit}.
	\end{proof}
	We remark that the above proof actually gives a stronger conclusion,
	of independent interest:
	given any linear ordering of the vertices in each of $S$,
	there is a path $P$ of $G$ such that $S\subset V(P)$
	and the vertices of $S$ lie on $P$ in the given order.
	
	We also need the following result~\cite{thn2022} on highly connected subgraphs of large chromatic number,
	which is an improvement on the constant factors of a result of Gir\~{a}o and Narayanan~\cite{girao2020subgraphs}.
	(In fact, it is proved in~\cite{thn2022} that the constant $4$ can be replaced by $3+\frac{1}{16}$,
	but we choose $4$ since this has a simple proof.)
	\begin{theorem}
		\label{thm:kappachi}
		For every integer $k\ge1$, every graph $G$ with $\chi(G)\ge4k$ contains a $k$-connected subgraph with chromatic number at least $\chi(G)-2k$.
	\end{theorem}
	Last but not least, we require the following result~\cite[Lemma 6.2]{delcourt2021reducing}, which will be helpful to construct models with fragments of reasonably small chromatic number.
	\begin{lemma}
		\label{lem:container}
		Let $G$ be a connected graph.
		Then for every $S\subset V(G)$,
		there is $S'\subset V(G)$ with $S\subset S'$ and $\abs{S'}\le3\abs{S}$
		and an connected induced subgraph $H$ of $G$ containing $S'$ such that $\chi(H\setminus S')\le2$. 
	\end{lemma}
	%\begin{theorem}
	%	\label{thm:knit}
	%	There is an integer $C=C_{\ref{thm:knit}}>0$ such that for every integer $k\ge1$,
	%	every $Ck$-connected graph is $k$-linked.
	%\end{theorem}
	
	\subsection{Rooted dense minors and dense wovenness}
	We now develop tools on rooted minors which will be necessary for our $(\vep,t)$-dense minors construction.
	Given graphs $G,H$ and $S\subset V(G)$ with $\abs{S}\le \abs{H}$,
	an \emph{$H$ model attached to $S$} in $G$ is an $H$ model $\mac M=\{M_1,\ldots,M_{\abs{H}}\}$ in $G$ with $\abs{S\cap V(M_i)}=1$ for all $i\in\{1,2,\ldots,\abs{S}\}$;
	thus $\mac M$ is an $H$ model rooted at $S$ in $G$ if $\abs{S}=\abs{H}$.
	A separation $(A,B)$ of $G$ is \emph{$S$-separation} if $S\subset A$;
	and $(A,B)$ \emph{avoids} a subset $D\subset V(G)$
	if $D\subset B\setminus A$.
	The following lemma (with a simple proof, cf.~\cite[Lemma 2]{MR1417341}) is an extension of an argument by Kawarabayashi~\cite{MR2278128},
	which in turn is based on earlier work of Robertson and Seymour~\cite{MR1309358}.
	\begin{lemma}
		\label{lem:rooted}
		Let $m,t\ge1$ and $n\ge0$ be integers with $m\ge n+2t$. 
		Let $G$ be a graph and $S\subset V(G)$ with $\abs{S}=t$.
		Let $D_1,\ldots,D_{m}\ne\emptyset$ be disjoint subsets of $V(G)$
		and $I:=\{i\in[m]:D_i\cap S=\emptyset\}$, such~that
		\begin{itemize}
			\item for every $i\in I$, $G[D_i]$ is connected;
			
			\item for every $j\in[m]\setminus I$,
			each component of $G[D_j]$ intersects $S$;
			
			\item for every $j\in[m]$, $D_j$ is anticomplete in $G$ to at most $n$ sets in $\mac D:=\{D_i:i\in I\}$; and
			
			\item there is no $S$-separation of $G$ of order less than $t$ avoiding more than $n$ sets in $\mac D$.
		\end{itemize}
		Then $G$ has an $H$ model attached to $S$ for some graph $H$ with $\abs{H}=m-t$ and $\Delta(\overline{H})\le n$.
	\end{lemma}
	\begin{proof}
		Suppose not.
		Let $G$ be a counterexample with $\abs{G}+\edge(G)$ minimal;
		let $D:=\bigcup_{j\in[m]}D_j$.
		If $E(G[S])=\emptyset$, then $G\setminus E(G[S])$ would be a smaller counterexample, a contradiction;
		so $E(G[S])=\emptyset$.
		\begin{claim}
			\label{claim:rooted}
			Every edge of $G$ is an edge between two distinct members of $D_1,\ldots,D_m$.
		\end{claim}
		\begin{subproof}
			Suppose there is $e\in E(G)$ not satisfying that property.
			For every $j\in[m]$, let $D_j'$ be obtained from $D_j$ after the contraction of $e$;
			then $D_1',\ldots,D_m'$ are nonempty disjoint connected subsets of $G/e$.
			By the minimality of $G$, $G/e$ has an $S$-separation of order less than $t$ avoiding more than $n$ sets in $\{D_i':i\in I\}$.
			Thus $G$ has an $S$-separation $(A,B)$ of order $t$ avoiding more than $n$ sets in $\mac D$
			such that $e\subset A\cap B=:S'$.
			Since $E(G[S])=\emptyset$,
			$S\ne S'$ which yields $B\subsetneq V(G)$.
			If there are fewer than $t$ vertex-disjoint paths from $S$ to $S'$ within $G[A]$,
			then \cref{thm:menger} would yield an $S$-separation $(A_1,A_2)$ of $G[A]$ of order less than~$t$;
			so $(A_1,A_2\cup B)$ would be an $S$-separation of $G$ of order less than $t$ avoiding more than $n$ sets in $\mac D$,
			a contradiction.
			Thus
			there are $t$ vertex-disjoint paths from $S$ to $S'$ within $G[A]$.
			If there is an $S'$-separation $(B_1,B_2)$ of $G[B]$ of order less than $t$ avoiding more than $n$ sets in $\mac D$,
			then $(A\cup B_1,B_2)$ would be an $S$-separation of $G$ of order less than $t$ avoiding more than $n$ sets in $\mac D$, a contradiction. 
			Now, note that for each $j\in[m]$, $D_j\cap B\ne\emptyset$ since $D_j$ is anticomplete to fewer than $n$ sets in $\mac D$;
			in particular either $G[D_j\cap B]$ is connected or its components intersect $S'$ .
			Thus, the minimality of $G$ applied to $G[B]$ with $S'$ and $\{D_j\cap B:j\in[m]\}$
			yields an $H$ model attached to $S'$ in $G[B]$ for some $H$ with $\abs{H}=m-t$ and $\Delta(\overline{H})\le n$.
			This model together with the above $t$ vertex-disjoint paths from $S$ to $S'$ yields an $H$ model attached to $S$ in $G$, a contradiction.
			This proves~\cref{claim:rooted}.
		\end{subproof}
		Now, by the minimality of $G$, it has no isolated vertex outside of $S$.
		So~\cref{claim:rooted} yields $D=V(G)$; and for every $j\in[m]$, if $\abs{D_j}\ge2$ then $D_j\subset S$.
		Let $T:=V(G)\setminus S=\bigcup_{i\in I}D_i$; then
		$\abs{T}\ge m-t\ge t+n$.
		If there are no $t$ disjoint edges between $S$ and $T$,
		then \cref{thm:menger} would yield an $S$-separation $(A,B)$ of order less than $t$ in $G$ with $T\subset B$ which avoids at least
		$\abs{T\setminus(A\cap B)}>t+n-t=n$
		sets in $\mac D$, a contradiction.
		Thus there are $t$ disjoint edges between $S$ and $T$;
		and contracting these edges would yield a graph $H$ with $\Delta(\overline{H})\le n$, a contradiction.
		This proves \cref{lem:rooted}.
	\end{proof}
	The following result is a straightforward consequence of \cref{lem:rooted}.
	\begin{corollary}
		\label{cor:rooted}
		Let $m,t\ge1$ and $n\ge0$ be integers with $m\ge n+2t$.
		Let $G$ be a $t$-connected graph, and let $S\subset V(G)$ with $\abs{S}=t$.
		If $G$ contains a minor $J$ with $\abs{J}=m$ and $\Delta(\overline{J})\le n$,
		then it contains an $H$ model attached to $S$ for some graph $H$ with $\abs{H}=m-t$ and $\Delta(\overline{H})\le n$.
	\end{corollary}
	For $\vep>0$ and integers $a\ge1,b\ge0$,
	a graph $G$ is \emph{$(\vep,a,b)$-woven} if for every subsets $R=\{r_1,\ldots,r_{a}\}$, $S=\{s_1,\ldots,s_b\}$,
	and $T=\{t_1,\ldots,t_b\}$ of $ V(G)$
	such that $s_i\ne t_j$ for all distinct $i,j\in[b]$,
	there is an $(\vep,a)$-dense graph $H$ such that
	$G$ has an $H$ model $\mac M$ 
	rooted at $R$ and a linkage $\mac P$ linking $\{(s_i,t_i):i\in[b]\}$
	such that $V(\mac M)\cap V(\mac P)=R\cap (S\cup T)$.
	Two remarks:
	\begin{itemize}
		\item the existence of $H$ depends entirely on $R$, $S$, and $T$ in this definition; and
		
		\item this definition can also be made for a much stronger notion of rooted minors which requires each fragment of the $H$ model to contain a prespecified vertex in $R$ (see~\cite{MR2407003}),
		but we do not need~it~here. 
	\end{itemize}
	
	The following useful lemma, an analogue of \cite[Lemma 5.15]{delcourt2021reducing},
	allows us to generate a dense minor and \dd weave\ee{} a given linkage around at the same time,
	as mentioned in \cref{sec:sketch}.
	\begin{lemma}
		\label{lem:woven}
		Let $\vep>0$, and let $a\ge1,b\ge0$ be integers.
		Let $G$ be a graph with an $(\vep,a,b)$-woven subgraph $F$.
		Let $S=\{s_1,\ldots,s_b\}$ and $T=\{t_1,\ldots,t_b\}$ be subsets of $V(G)$ such that $s_i\ne t_j$ for all distinct $i,j\in[b]$. 
		Let $\mac P'$ be a linkage in $G$ linking $\{(s_i,t_i):i\in[b]\}$,
		and let $\{r_1,\ldots,r_{a}\}\subset V(F)$.
		Then there is an $(\vep,a)$-dense graph $H$ such that there is an $H$ model $\mac M$ rooted at $\{r_1,\ldots,r_{a}\}$ in $F$
		and a linkage $\mac P$ in $G$ linking $\{(s_i,t_i):i\in[b]\}$ such that
		$V(\mac M)\cap V(\mac P)\subset R\cap(S\cup T)$ 
		and $V(\mac P)\subset V(F)\cup V(\mac P')$.
	\end{lemma}
	\begin{proof}
		For each $i\in[a]$, let $P_i'$ be the path in $\mac P'$ linking $s_i,t_i$.
		Let $I:=\{i\in[b]:V(P_i')\cap V(F)\ne\emptyset\}$.
		For every $i\in I$,
		let $s_i',t_i'$, respectively, be the vertices nearest to $s_i,t_i$ on $P_i'$.
		Since $F$ is $(\vep,a,b)$-woven,
		there exists an $(\vep,a)$-dense graph $H$ and a linkage $\mac Q=\{Q_i:i\in I\}$ in $F$ linking $\{(s_i',t_i'):i\in I\}$
		such that there is an $H$ model $\mac M$ rooted at $\{r_1,\ldots,r_a\}$ in $F$ which is vertex-disjoint from $\mac Q$.
		For every $i\in I$, let $P_i:=s_iP_i's_i'Q_it_i'P_i't_i$;
		then $\mac M$ together with the linkage $\mac P=\{P_i:i\in I\}\cup \{P_i':i\in[b]\setminus I\}$
		satisfy the conclusion of \cref{lem:woven}.
	\end{proof}
	We now give two applications of \cref{cor:rooted} which will be needed later on.
	The first one says that linear connectivity together with a very dense minor yield dense wovenness.
	\begin{lemma}
		\label{lem:woven1}
		For every $\vep>0$ and integer $a\ge2$,
		every $8a$-connected graph $G$ containing a $(\frac{1}{256}\vep,32a)$-dense minor is $(\vep,a,3a)$-woven.
	\end{lemma}
	\begin{proof}
		Since $G$ has a $(\frac{1}{256}\vep,32a)$-dense minor, it has a minor $J$ with $16a\le\abs{J}\le 32a$ and $\Delta(\overline{J})\le \frac{1}{2}\vep a$.
		Let $b:=3a$; let $R=\{r_1,\ldots,r_a\}$, $S=\{s_1,\ldots,s_b\}$, and $T=\{t_1,\ldots,t_b\}$ be subsets of $V(G)$ such that $s_i=t_j$ for some $i,j\in[b]$ only if $i=j$.
		We may assume $S\cap T=\emptyset$.
		Since $G$ is $8a$-connected and so has minimum degree at least $8a$,
		there is $R'=\{r_1',\ldots,r_a'\}\subset V(G)\setminus(R\cup S\cup T)$
		with $r_ir_i'\in E(G)$ for all $i\in[a]$.
		Now, let $G':=G\setminus R$;
		then $G'$ contains a minor $J'$ 
		which is an induced subgraph of $J$ 
		with $15a\le\abs{J'}\le 32a$ and $\Delta(\overline{J'})\le\Delta(\overline{J})\le\frac{1}{2}\vep a$.
		Let $m:=\abs{J'}$, $n:=\Delta(\overline{J'})$,
		and $t:=a+2b=7a$;
		then $m-n\ge (15-\vep)a\ge 14a=2t$.
		So by~\cref{cor:rooted},
		$G'$ has an $F$ model $\{G_1,\ldots,G_{m-t}\}$ for some $F$ with
		\begin{itemize}
			\item $V(F)=[m-t]$ and $\Delta(\overline{F})\le \Delta(\overline{J'})\le \frac{1}{2}\vep a$;
			
			\item $r_i'\in V(G_i)$ for all $i\in[a]$; and
			
			\item $s_i\in V(G_{a+i})$ and $t_i\in V(G_{a+2i})$ for all $i\in[b]$.
		\end{itemize}
		
		Observe that every two vertices of $F$ have at least $m-t-1-\vep a$ common neighbors,
		and so at least $m-2t-1-\vep a=m-(10+\vep)a-1\ge 3a=b$ common neighbors in $V(F)\setminus[t]$.
		It follows that there are distinct $j_1,\ldots,j_b\in V(F)\setminus[t]$ such that for every $i\in[b]$, $j_i$ is adjacent to both $a+i$ and $a+2i$.
		
		Now, let $H$ be the subgraph of $F$ induced by $[a]$; then $\abs{H}=a$ and $\Delta(\overline{H})\le\frac{1}{2}\vep a\le\vep(a-1)$,
		in particular $H$ is $(\vep,a)$-dense. 
		For every $i\in[b]$, let $P_i$ be some path between $s_i$ and $t_i$ in the subgraph of $G'$ induced by $V(G_{a+i})\cup V(G_{j_i})\cup V(G_{a+2i})$;
		then $\{G_1,\ldots,G_a\}$ is an $H$ model in $G'$ rooted at $R'$ and is vertex-disjoint from the linkage $\{P_1,\ldots,P_b\}$ linking $\{(s_i,t_i):i\in[b]\}$ in $G'$.
		Combining $\{G_1,\ldots,G_a\}$ with the edges $\{r_ir_i':i\in[a]\}$ gives an $H$ model in $G$ rooted at $R$ and vertex-disjoint from $\{P_1,\ldots,P_b\}$. This proves~\cref{lem:woven1}.
	\end{proof}
	Note the same argument shows that for every integer $k\ge1$, there is $K>0$ such that for every $\vep>0$ and integer $a\ge2$,
	every $Ka$-connected graph containing a $(K^{-1}\vep,Ka)$-dense minor is $(\vep,a,ka)$-woven.
	
	We also need the following result of Bollob\'{a}s and Thomason~\cite[Lemma 1]{MR1417341} (cf. \cref{lem:mader,lem:denseminor}).
	\begin{lemma}
		\label{lem:dense}
		Let $d\ge3$ be an integer.
		Then every graph with average degree at least $d$ contains a minor $J$ with $\abs{J}\le \frac{1}{2}d$ and $\Delta(\overline{J})\le\frac{1}{2}\abs{J}-\frac{3}{40}d$.
	\end{lemma}
	Now we come the second application of \cref{cor:rooted}, which implies that high connectivity yields dense wovenness.
	\begin{lemma}
		\label{lem:woven2}
		There is an integer $C=C_{\ref{lem:woven2}}>10$ such that the following holds.
		Let $a\ge1$ and $b\ge0$ be integers, and let $\vep\in(0,\frac{1}{3})$.
		Then every $(2a+2b)$-connected graph $G$ with average degree at least $d= \ceil{C\max(a\sqrt{\log(1/\vep)},b)}$
		is $(\vep,a,b)$-woven.
	\end{lemma}
	\begin{proof}
		The constant $C$ is chosen implicitly to satisfy the inequalities throughout the proof.
		
		Let $R=\{r_1,\ldots,r_a\}$, $S=\{s_1,\ldots,s_b\}$, and $T=\{t_1,\ldots,t_b\}$ be subsets of $V(G)$ such that $s_i=t_j$ for some $i,j\in[b]$ only if $i=j$.
		We may assume $S\cap T=\emptyset$.
		Since $G$ has minimum degree at least $2a+2b$,
		there is $R'=\{r_1',\ldots,r_i'\}\subset V(G)\setminus (R\cup S\cup T)$
		with $r_ir_i'\in E(G)$ for all $i\in[a]$.
		Let $G':=G\setminus R$;
		then $G'$ is $(a+2b)$-connected and has average degree at least $d':=d-a$.
		By \cref{lem:dense}, $G'$ has a minor $J$ with $\abs{J}\le \frac{1}{2}d'$ and $\Delta(\overline{J})\le\frac{1}{2}\abs{J}-\frac{3}{40}d'$.
		Let $m:=\abs{J}$, $n:=\Delta(\overline{J})$, and $t:=a+2b$;
		then $m-n\ge\frac{1}{2}m+\frac{3}{40}d'\ge 2t$.
		So by~\cref{cor:rooted},
		$G'$ has an $F$ model $\{G_1,\ldots,G_{m-t}\}$ for some graph $F$ with 
		\begin{itemize}
			\item $V(F)=[m-t]$ and $\Delta(\overline{F})\le \Delta(\overline{J})
			\le\frac{1}{2}m-\frac{3}{40}d'$;
			
			\item $r_i'\in V(G_i)$ for all $i\in[a]$; and
			
			\item $s_i\in V(G_{a+i})$ and $t_i\in V(G_{a+2i})$ for all $i\in[b]$.
		\end{itemize}
		Note that every two vertices of $F$ have at least
		$m-t-1-2(\frac{1}{2}m-\frac{3}{40}d')=\frac{3}{40}d'-t-1$ common neighbors,
		and so at least $\frac{3}{40}d'-2t-1=\frac{3}{40}(d-a)-2t-1\ge \frac{1}{20}d$ common neighbors in $V(F)\setminus[t]$.
		
		%	Put $p:=m-2t\ge n$, and let $\mac B':=\{B_{t+1},\ldots,B_{t+p}\}$.
		Let $I$ be a random subset of $V(F)\setminus[t]$ where each element is included independently with probability $1/3$.
		Then by Markov's inequality, $\abs{I}\le \frac{1}{2}(\abs{F}-t)$ with probability at least $1/3$.
		By Hoeffding's inequality, two given vertices of $F$ have at most $\frac{1}{80}d$ common neighbors in $I$ with probability at most $e^{-d/3000}$.
		So by a union bound,
		with probability at least $1-d^2e^{-d/3000}>2/3$,
		every two vertices of $F$ have at least $\frac{1}{80}d$ common neighbors in $I$.
		Hence there is a choice of $I$ such that $\abs{I}\le\frac{1}{2}(\abs{F}-t)$
		and every two vertices of $F$ have at least $\frac{1}{80}d$ common neighbors in $I$.
		
		Let $L:=F\setminus (I\cup[t])$; then $\abs{L}
		\ge\abs{F}-\frac{1}{2}(\abs{F}+t)
		=\frac{1}{2}(\abs{F}-t)
		=\frac{1}{2}m-t$, and so
		\[\delta(L)\ge\abs{L}-1-\Delta(\overline{F})
		\ge \frac{1}{2}m-t-1-\left(\frac{1}{2}m-\frac{3}{40}d'\right)
		= \frac{3}{40}(d-a)-t-1\ge\frac{1}{20}d
		\ge C_{\ref{thm:denselinearhwg}}a\sqrt{\log(1/\vep)}.\]
		Thus by~\cref{thm:denselinearhwg}, $L$ contains an $H$ model $\{L_1,\ldots,L_{a}\}$ for some $(\vep,a)$-dense graph $H$.
		
		Now, 
		%	let $\pi\colon [a]\to V(H)$ be a bijection.
		since every two vertices of $F$ (thus $L$) have at least $\frac{1}{80}d\ge a+b$ common neighbors in $I$,
		there exist distinct $j_1,\ldots,j_a,k_1,\ldots,k_b\in I$ such that in $F$,
		\begin{itemize}
			\item for every $i\in[a]$, $j_i$ is adjacent to $i$ and has a neighbor in $V(L_{i})$ for all $i\in[a]$; and
			
			\item for every $i\in[b]$, $k_i$ is adjacent to $a+i$ and $a+2i$.
		\end{itemize}
		For every $i\in[a]$, let
		$M_i'$ be the subgraph of $G'$ induced by $V(G_i)\cup V(G_{j_i})\cup\bigcup_{j\in V(L_{i})}V(G_j)$;
		and for every $i\in[b]$,
		let $P_{i}$ be some path between $s_{i}$ and $t_{i}$ in the subgraph of $G'$ induced by $V(G_{a+i})\cup V(G_{k_i})\cup V(G_{a+2i})$. 
		Then $\{M_1',\ldots,M_a'\}$ is an $H$ model rooted at $R'$ in $G'$
		and is vertex-disjoint from the linkage $\{P_1,\ldots,P_b\}$ linking $\{(s_i,t_i):i\in[b]\}$.
		Let $M_i:=M_i'\cup{r_ir_i'}$ for every $i\in[a]$;
		then $\{M_1,\ldots,M_a\}$ is an $H$ model rooted at $R$ in $G$ which is vertex-disjoint from $\{P_1,\ldots,P_b\}$.
		This proves~\cref{lem:woven2}.
	\end{proof}
	\section{Chromatic-inseparability}
	\label{sec:insep}
	This section presents the proof of \cref{thm:insep}.
	We need more definitions.
	For graphs $G,H$ and an $H$ model $\mac M=\{M_1,\ldots,M_{k}\}$ in $G$ where $k=\abs{H}\ge1$,
	a subset $S$ of $V(G)$ is a \emph{core} of $\mac M$ if for all distinct $i,j\in[k]$,
	$S\cap V(M_i)$ and $S\cap V(M_j)$ are not anticomplete in $G$ if $ij\in E(H)$.
	For $U\subset V(\mac M)$ such that $\mac M$ is rooted at $U$, a subgraph $F$ of $G$ is \emph{tangent to $\mac M$ at $U$} if $ V(F)\cap V(M_i)=U\cap V(M_i)$
	for all $i\in[k]$.
	The following lemma formalizes the iterative construction of an $(\vep,t)$-dense minor.
	\begin{lemma}
		\label{lem:chi-insep}
		There is an integer $C=C_{\ref{lem:chi-insep}}>0$ such that the following holds.
		Let $\vep\in(0,1/2)$ and $\Gamma:=\log(1/\vep)$.
		Let $t\ge\Gamma^{1/2}$ be an integer, let $G$ be a graph, and let
		\[g(G,\vep,t)=g_{\ref{lem:chi-insep}}(G,\vep,t)
		=1+\max_{F\subset G}
		\left\{\frac{\chi(F)}{t}:
		\abs{F}\le C^2\Gamma^4t,\,
		\text{$F$ is $(\vep,t)$-dense-minor-free}\right\}.\]
		If $G$ is $(C^2t\cdot g(G,\vep,t))$-chromatic-inseparable
		and $\chi(G)\ge2C^2t\cdot g(G,\vep,t)$,
		then for every integer $q$ with $0\le q\le\ceil{\Gamma^{1/2}}$, $G$ contains an $(\vep,t)$-dense minor
		or contains both of the following
		\begin{itemize}
			\item an $H$ model $\mac A$ rooted at $U$ with a core $S\subset V(G)$ where $U\subset S$, $\abs{S}\le q^2C_{\ref{thm:smalldense}}^2\Gamma^3t$, and $H$ is an $(\vep,s)$-dense graph with $s=q\ceil{\Gamma^{-1/2}t}$; and
			
			\item a $Ct$-connected subgraph $F$ tangent to $\mac A$ at $U$ where $\chi(F)\ge\chi(G)-\frac{1}{2}C^2t\cdot g(G,\vep,t)$.
		\end{itemize} 
	\end{lemma}
	\begin{proof}
		The constant $C$ is chosen implicitly to satisfy the inequalities throughout the proof.
		
		Induction on $q$.
		For $q=0$, the subgraph $F$ exists by~\cref{thm:kappachi}.
		So we may assume $q\ge1$ and the lemma holds for $q-1$.
		Let $r:=\ceil{\Gamma^{-1/2}t}\ge1$
		and $s':=s-r=(q-1)r$;
		then $s'\le \Gamma^{1/2}(\Gamma^{-1/2}t+1)\le 2t$ since $q-1\le\Gamma^{1/2}\le t$.
		We may assume $G$ is $(\vep,t)$-dense-minor-free;
		then by induction, $G$ contains
		\begin{itemize}
			\item an $H'$ model $\mac A'=\{A_1',\ldots,A_{s'}'\}$ rooted at $U'$ with a core $S'\subset V(G)$ where $U'\subset S'$,
			$\abs{S}\le (q-1)^2C_{\ref{thm:smalldense}}^2\Gamma^3t$, and $H$ is an $(\vep,s')$-dense graph; and
			
			\item a $Ct$-connected subgraph $F'$ tangent to $\mac A'$ at $U'$ where $\chi(F')\ge\chi(G)-\frac{1}{2}C^2t\cdot g(G,\vep,t)$.
		\end{itemize}
		For every $i\in[s']$, let $u_i$ be the unique vertex in $V(F')\cap A_i'$;
		then $U'=\{u_1,\ldots,u_{s'}\}$.
		Let $k:=Ct$;
		then
		\begin{align*}
			\chi(F'\setminus U')
			\ge\chi(F')-s'
			&\ge \chi(G)-\frac{1}{2}C^2t\cdot g(G,\vep,t)-2t\\
			&\ge C^2t\cdot g(G,\vep,t)
			\ge C^2t\cdot g_{\ref{cor:smalldense}}(G,\vep,t) \ge 3C_{\ref{thm:smalldense}}^2k.
		\end{align*}
		By~\cref{cor:smalldense},
		$F'\setminus U'$ has $q$ vertex-disjoint $k$-connected subgraphs $F_0,F_1,\ldots,F_{q-1}$
		with $\abs{F_i}\le C_{\ref{thm:smalldense}}^2\Gamma^3t$ for all $i\in\{0,1,\ldots,q-1\}$.
		For every $i\in[q-1]$,
		let $U_i\subset V(F_i)$ with $\abs{U_i}=2r$,
		and let $T:=U'\cup\bigcup_{i=1}^{q-1}U_i$;
		then $\abs{T}=s'+2r(q-1)=3s'$.
		Since $\kappa(F')\ge Ct\ge 6s'$, by \cref{thm:menger} and by duplicating each vertex of $T$,
		we get a set $\mac P_1$ of $6s'$ paths in $F'$ between $T$ and $V(F_0)$
		which pairwise share no vertex in $V(F')\setminus T$,
		such that each vertex of $T$ is the endpoint of exactly two paths in $\mac P_1$.
		Let $D:=\bigcup_{i=0}^{q-1}V(F_i)$;
		then
		\begin{equation}
			\label{eq:chi-insep1}
			\abs{D}\le\abs{F_0}+\abs{F_1}+\cdots+\abs{F_{q-1}}
			\le q\cdot C_{\ref{thm:smalldense}}^2\Gamma^3t
			\le C^2\Gamma^4t
		\end{equation}
		so $\chi(F[D])=\chi(G[D])\le t\cdot g(G,\vep,t)$ by the definition of $g(G,\vep,t)$.
		It follows that
		\[\chi(F'\setminus D)\ge\chi(F')-t\cdot g(G,\vep,t)
		\ge C^2t\cdot g(G,\vep,t)\ge 4k,\]
		and so by~\cref{thm:kappachi},
		$F'\setminus D$ contains a $k$-connected subgraph $G_1$ with
		\begin{equation}
			\label{eq:chi-insep2}
			\chi(G_1)\ge\chi(F'\setminus D)-2k
			\ge \chi(F')-t\cdot g(G,\vep,t)-2Ct
			\ge \chi(G)-C^2t\cdot g(G,\vep,t).
		\end{equation}
		Since $\kappa(F'\setminus T)\ge \kappa(F')-\abs{U}\ge Ct-3s'\ge 2s$,
		\cref{thm:menger} yields a set $\mac P_2$
		of $2s$ vertex-disjoint paths between $V(F_0)$ and $V(G_1)$ in $F'\setminus T$.
		By~\cref{lem:redundant} applied to $\mac P_1,\mac P_2$ and by identifying in pairs the endpoints in $V(G_1)$ of the paths of $\mac P_2$,
		we get a set $\mac P'$ of $3s'+s$ vertex-disjoint paths between $T\cup V(G_1)$ and $V(F_0)$ such that
		each vertex in $T$ is the endpoint of exactly one path in $\mac P'$.
		For each $i\in[q-1]$, since
		$2r\log^{1/2}(2/\vep)\le 4r\Gamma^{1/2}
		\le 8t$ and $3s'+s\le 10t$, we have that
		\[\kappa(F_i)
		\ge k=Ct \ge C_{\ref{lem:woven2}}\cdot 10t\ge C_{\ref{lem:woven2}}\cdot\max(2r\log^{1/2}(2/\vep),3s'+s),\]
		so by~\cref{lem:woven2}, $F_i$ is $(\frac{1}{2}\vep,2r,3s'+s)$-woven.
		Thus, by applying~\cref{lem:woven} iteratively to $F_1,\ldots,F_{q-1}$,
		we get a set $\mac P$ of $3s'+s$ vertex-disjoint paths between $T\cup V(G_1)$ and $V(F_0)$ such that
		\begin{itemize}
			\item each vertex in $T$ is the endpoint of exactly one path in $\mac P$; and
			
			\item for every $i\in[q-1]$, $F_i$ has an $H_i$ model $\mac M_i$ rooted at $U_i$ for some $(\frac{1}{2}\vep,2r)$-dense graph $H_i$
			such that $V(\mac M_i)\cap V(\mac P)= U_i$.
		\end{itemize}
		
		For every $u\in T\cup V(G_1)$,
		let $P_u$ be the path of $\mac P$ having $u$ as an endpoint.
		For every $i\in[q-1]$, for each $u\in U_i$, let $M_u$ be the member of $\mac M_i$ containing $u$;
		and also, let $X_i\cup Y_i$ be a partition of $U_i$ with $\abs{X_i}=\abs{Y_i}=r$,
		and let $Y_i=\{y_{j}^i:j\in[r]\}$. 
		Let $\bigcup_{i=1}^{q-1}X_i=\{x_1,\ldots,x_{s'}\}$.
		Let $U=\{v_1,\ldots,v_s\}$ be the vertices in $V(G_1)$ that are the endpoints of some paths in $\mac P$.
		For every $i\in[s']$, let $V_i$ be the set of three vertices in $V(F_0)$ linked to $\{u_i,x_i,v_i\}$ via $\mac P$;
		and for every $j\in[r]$,
		let $W_j$ be the set of $q$ vertices in $V(F_0)$ linked to $\{y_j^i:i\in[q-1]\}\cup\{v_{s'+j}\}$ via $\mac P$.
		Since $\kappa(F_0)\ge k=Ct\ge 11(3s'+s)$,
		\cref{thm:knit} implies that $F_0$ is $(3s'+s,s)$-knit,
		and so $F_0$ contains vertex-disjoint connected subgraphs $J_1,\ldots,J_s$ such that
		$V_i\subset V(J_i)$ for all $i\in[s']$
		and $W_j\subset V(J_{s'+j})$ for all $j\in[r]$. Now, let
		\begin{itemize}
			\item $B_i:=(A_i'\cup P_{u_i})\cup (M_{x
				_i}\cup P_{x_i})\cup J_i\cup P_{v_i}$ for every $i\in[s']$; and
			
			\item $B_{s'+j}:=J_{s'+j}\cup P_{v_{s'+j}}\cup\bigcup_{i\in[q-1]}(M_{y_j^i}\cup P_{y_j^i})$ for all $j\in[r]$.
		\end{itemize}
		Then $\mac B:=\{B_1,\ldots,B_s\}$ is an $H$ model with $S:=S'\cup D\cup U$ as a core, for some $H$ with $\abs{H}=s$.
		By~\eqref{eq:chi-insep1},
		\begin{equation}
			\label{eq:chi-insep3}
			\abs{S}=\abs{S'}+\abs{D}+\abs{U}
			\le
			(q-1)^2C_{\ref{thm:smalldense}}^2\Gamma^3t+q\cdot C_{\ref{thm:smalldense}}^2\Gamma^3t+s
			\le q^2C_{\ref{thm:smalldense}}^2\Gamma^3t.
		\end{equation}
		\begin{claim}
			\label{claim:chi-insep1}
			$H$ is $(\vep,s)$-dense.
		\end{claim}
		\begin{subproof}
			Since $\{A_1',\ldots,A_{s'}'\}$ is an $H'$ model in $G$ and $H'$ is $(\vep,s')$-dense,
			there are at most $\vep{s'\choose2}$ anticomplete pairs among $\{B_i:i\in[s']\}$.
			Moreover, the number of anticomplete pairs of the form $\{V(B_i),V(B_{s'+j})\}$ for $i\in[s'],j\in[r]$
			plus the number of anticomplete pairs among $\{V(B_{s'+j}):j\in[r]\}$
			is at most $\sum_{i=1}^{q-1}\edge(\overline{H_i})\le \frac{1}{2}(q-1)\vep{2r\choose2}$,
			since each of $H_1,\ldots,H_{q-1}$ is $(\frac{1}{2}\vep,2r)$-dense
			and every nonedge of each $H_i$ contributes to at most one such anticomplete pair.
			Thus, in order to show that $H$ is $ (\vep,s)$-dense, it suffices to prove
			${s'\choose2}+\frac{1}{2}(q-1){2r\choose2}\le{s\choose2}$, which holds since (recall that $s=s'+r$)
			\[{s\choose2}-{s'\choose2}
			=\frac{1}{2}(s-s')(s+s'-1)
			\ge rs'=(q-1)r^2
			\ge \frac{1}{2}(q-1){2r\choose2}.
			\qedhere\]
		\end{subproof}
		Note that while $G_1$ is $k$-connected and tangent to $\mac B$ at $U$,
		it only satisfies $\chi(G_1)\ge\chi(G)- C^2t\cdot g(G,\vep,t)$ by~\eqref{eq:chi-insep2}.
		For every $i\in[s]$, by~\cref{lem:container},
		there is an induced connected subgraph $A_i$ of $G[V(B_i)]$ and $Z_i\subset V(A_i)$
		with $S\cap V(B_i)\subset Z_i$, $\abs{Z_i}\le 3\abs{S\cap V(B_i)}$,
		and $\chi(A_i\setminus Z_i)\le2$.
		Then $\mac A=\{A_1,\ldots,A_s\}$ is an $H$ model in $G$ with $S$ as a core such that $G_1$ is tangent to $\mac A$ at $U\subset S\cap V(\mac A)$. 
		Let $Z:=\bigcup_{i=1}^sZ_i$; then
		\[\abs{Z}
		\le3\abs{S}
		\le 3q^2C_{\ref{thm:smalldense}}^2\Gamma^3t\le C^2\Gamma^4t,\]
		where the second inequality holds by~\eqref{eq:chi-insep3}, and so $\chi(G[Z])\le t\cdot g(G,\vep,t)$.
		Let $A:=G[\bigcup_{i=1}^sV(A_i)]$;
		then
		\begin{equation*}
			%		\label{eq:chi-insep3}
			\chi(A)\le2s+\chi(G[Z])
			\le 2s+t\cdot g(G,\vep,t)
			\le 7t\cdot g(G,\vep,t).
		\end{equation*}
		It follows that
		\[\chi(G\setminus A)\ge\chi(G)-\chi(A)
		\ge C^2t\cdot g(G,\vep,t)-7t\cdot g(G,\vep,t)
		\ge 4Ct,\]
		and so by~\cref{thm:kappachi},
		$G\setminus A$ has a $k$-connected subgraph $G_2$ with
		\[\chi(G_2)\ge\chi(G\setminus A)-2k
		\ge \chi(G)-7t\cdot g(G,\vep,t)-2Ct
		\ge \chi(G)-\frac{1}{2}C^2t\cdot g(G,\vep,t).\]
		Thus, if $\abs{V(G_1)\cap V(G_2)}\le k$, then
		\[\chi(G_2\setminus G_1)
		\ge \chi(G)-7t\cdot g(G,\vep,t)-2Ct-k
		\ge \chi(G)-C^2t\cdot g(G,\vep,t),\]
		and so $G_1$ and $G_2\setminus G_1$ would be two vertex-disjoint subgraphs of $G$ each with chromatic number at least $\chi(G)-C^2t\cdot g(G,\vep,t)$,
		contradicting that $G$ is $(C^2t\cdot g(G,\vep,t))$-chromatic-inseparable.
		Hence $G_1,G_2$ have more than $k$ common vertices,
		which implies that $F:=G_1\cup G_2$ is a $k$-connected subgraph tangent to $\mac A$ at $U$ with
		$\chi(F)\ge\chi(G_2)\ge \chi(G)-\frac{1}{2}C^2t\cdot g(G,\vep,t)$.
		This proves~\cref{lem:chi-insep}.
	\end{proof}
	We are now ready to prove \cref{thm:insep},
	which we restate here (in a slightly different form) for the convenience of the readers.
	\begin{theorem}
		\label{thm:chi-insep}
		There is an integer $C>0$ such that the following holds.
		Let $\vep\in(0,1/2)$ and $\Gamma:=\log(1/\vep)$. 
		Let $t\ge\Gamma^{1/2}$ be an integer, let $G$ be a graph, and let
		\[g(G,\vep,t)
		:=1+\max_{F\subset G}\left\{\frac{\chi(F)}{t}:\abs{F}\le C\Gamma^4t,\,
		\text{$F$ is $(\vep,t)$-dense-minor-free}\right\}.\]
		If $G$ is $(Ct\cdot g(G,\vep,t))$-chromatic-inseparable
		and $\chi(G)\ge 2Ct\cdot g(G,\vep,t)$, then $G$ has a $(\vep,t)$-dense-minor.
	\end{theorem}
	\begin{proof}
		Let $C:=C_{\ref{lem:chi-insep}}^2$.
		\cref{thm:chi-insep} follows from~\cref{lem:chi-insep} with $q=\ceil{\Gamma^{1/2}}$,
		as ${s=q\ceil{\Gamma^{-1/2}t}\ge t}$.
	\end{proof}
	\section{Finishing the proof}
	\label{sec:proof}
	This section deals with the general case.
	The recursive construction of an $(\vep,t)$-dense minor will be made rigorous by the following lemma.
	\begin{lemma}
		\label{lem:chisepwoven}
		There is an integer $C=C_{\ref{lem:chisepwoven}}>0$ such that the following holds.
		Let $\vep\in(0,\frac{1}{256})$, and let $\Gamma:=\log(1/\vep)$.
		Let $t\ge\Gamma^2$ be an integer, let $G$ be a graph, and let
		\[f(G,\vep,t)=f_{\ref{lem:chisepwoven}}(G,\vep,t)
		:=1+\max_{F\subset G}\left\{\frac{\chi(F)}{a}:
		\Gamma^{-1/2}t\le a\le t,\,
		\abs{F}\le C\Gamma^4a,\,
		\text{$F$ is $(\vep,a)$-dense-minor-free}\right\}.\]
		If $t$ is a power of $3$, $s$ is an integer with $0\le s\le\log_3t$, and $G$ is $Ca$-connected with $a=(\frac{2}{3})^st$ and
		\[\chi(G)\ge C(t+11a\cdot f(G,\vep,t)),\]
		then for $\xi=\log_{\frac{3}{2}}(\frac{4}{3})=0.7095\ldots$,
		\begin{itemize}
			\item if $s\ge\frac{1}{2}\log_{\frac{3}{2}}\Gamma$, then $G$ is $(\vep,a,3a)$-woven; and
			
			\item if $s\le\frac{1}{2}\log_{\frac{3}{2}}\Gamma$,
			then $G$ is $((\frac{3}{4})^s\Gamma^{\xi}\vep,a,3a)$-woven.
		\end{itemize}
		In particular, $G$ is always $(\Gamma\vep,a,3a)$-woven.
	\end{lemma}
	\begin{proof}
		The constant $C$ is chosen implicitly to satisfy the inequalities throughout the proof.
		
		Since $t\ge\Gamma^2$, $\log_3t\ge 2\log_3\Gamma> \frac{1}{2}\log_{3/2}\Gamma$.
		We proceed by backward induction~on~$s$.
		If $\frac{1}{2}\log_{\frac{3}{2}}\Gamma\le s\le \log_3t$
		then $a\le\Gamma^{-1/2}t$.
		Since $\chi(G)\ge Ct+1$, by \cref{lem:avgkappa},
		$G$ contains a subgraph with connectivity at least $\frac{1}{4}Ct \ge C_{\ref{lem:woven2}}t\ge C_{\ref{lem:woven2}}\Gamma^{1/2}a$,
		thus is $(\vep,a,3a)$-woven by~\cref{lem:woven2}.
		
		We may thus assume that $s<\frac{1}{2}\log_{\frac{3}{2}}\Gamma$
		and that the lemma is true for $s+1$;
		note that $(\frac{3}{4})^{s+1}\Gamma^{\xi}
		>\frac{3}{4}\Gamma^{\xi/2}>1$.
		Put $\vep':=\Gamma^{\xi}\vep$.
		Let $b:=3a$, and let $R=\{r_1,\ldots,r_a\}$, $S=\{s_1,\ldots,s_b\}$,
		$T=\{t_1,\ldots,t_b\}$ be subsets of $V(G)$
		such that $s_i=t_j$ for some $i,j\in[b]$ only if $i=j$;
		we may assume $R,S,T$ are pairwise disjoint.
		Let $Z:=R\cup S\cup T$,
		and let $G_1:=G\setminus Z$;
		then $\kappa(G_1)\ge(11C-7)a\ge 8a$,
		in particular $G_1$ has average degree at least $(11C-7)a\ge C_{\ref{thm:smalldense}}a$.
		If $G_1$ has a $(\frac{1}{256}\vep,32a)$-dense minor,
		then by~\cref{lem:woven1}, it is $(\vep,a,3a)$-woven and we are done.
		Thus we may assume $G_1$ is $(\frac{1}{256}\vep,32a)$-dense-minor-free;
		and so by~\cref{thm:smalldense} with $k=40a$,
		$G_1$ has a subgraph $F_0$ with  $\kappa(F_0)\ge 40a$ and (note that $\vep<\frac{1}{256}$)
		\[\abs{F_0}\le C_{\ref{thm:smalldense}}^2\cdot32a\cdot \log^3(256/\vep)
		\le 256C_{\ref{thm:smalldense}}^2a\cdot\log^3(1/\vep)
		=256C_{\ref{thm:smalldense}}^2\Gamma^3a
		\le C\Gamma^4a.\]
		Since $a=(\frac{2}{3})^st\ge\Gamma^{-1/2}t$,
		$\chi(F_0)\le a\cdot f(G,\vep,t)$ by the definition of $f(G,\vep,t)$.
		Since $\kappa(G)\ge Ca\ge 14a$, \cref{thm:menger} yields a set $\mac P_1$ of $14a$ paths in $G$ between $Z$ and $V(F_0)$ which pairwise share no vertex in $V(G)\setminus Z$,
		such that each vertex of $Z$ is the endpoint of exactly two paths in $\mac P_1$.
		We may assume the paths in $\mac P_1$ are induced paths in $G$;
		and so $\chi(G[V(\mac P_1)])\le 28a$.
		Let $F:=G_1\setminus(V(F_0)\cup V(\mac P_1))$; then
		\begin{align*}
			\chi(F)\ge \chi(G_1)-\chi(G[V(F_0)\cup V(\mac P_1)])
			&\ge \chi(G)-35a-a\cdot f(G,\vep,t)\\
			&\ge Ct+(11C-36)a\cdot f(G,\vep,t)
			\ge 4Ca.
		\end{align*}
		By~\cref{thm:kappachi}, $F$ contains a $Ca$-connected subgraph $G'$ with
		\begin{equation}
			\label{eq:reducing1}
			\chi(G')\ge \chi(F)-2Ca
			\ge Ct+(9C-36)a\cdot f(G,\vep,t).
		\end{equation}
		\begin{claim}
			\label{claim:woven}
			$G'$ is $((\frac{3}{4})^s\vep',a,0)$-woven.
		\end{claim}
		\begin{subproof}
			Let $R'=\{r_1',\ldots,r_a'\}\subset V(G')$;
			we need to show there is an $H$ model in $G'$ rooted at $R'$ for some $((\frac{3}{4})^s\vep',a)$-dense graph $H$.
			Let $G_0:=G'\setminus R'$.
			Since $\kappa(G')\ge Ca$,
			there is $U=\{u_1,\ldots,u_{2a}\}\subset V(G_0)$
			with $r_i'$ adjacent to $u_i$ and $u_{a+i}$ for all $i\in[a]$.
			Let $G_0':=G_0\setminus U$.
			Put $m:=C_{\ref{thm:chi-insep}}\cdot32a\cdot g_{\ref{thm:chi-insep}}(G,\frac{1}{256}\vep,32a)$.~Since \[C_{\ref{thm:chi-insep}}\cdot32a\cdot\log^4(256/\vep)
			\le 512C_{\ref{thm:chi-insep}}a\cdot \log^4(1/\vep)
			\le C\Gamma^4a,\]
			we have $f(G,\vep,t)\ge g_{\ref{thm:insep}}(G,\frac{1}{256}\vep,32a)$;
			so by~\eqref{eq:reducing1},
			\[\chi(G_0')\ge \chi(G')-\abs{R'\cup U}
			=\chi(G')-3a\ge Ct+(9C-39)a\cdot f(G,\vep,t)
			\ge 3m.\]
			Since $G_0'$ is $(\frac{1}{256}\vep,32a)$-dense-minor-free, and since
			$32a\ge 32\Gamma^{-1/2}t
			\ge 32\Gamma^{1/2}
			\ge \log^{1/2}(256/\vep)$,
			\cref{thm:chi-insep} implies that
			$G_0'$ is $m$-chromatic-separable.
			Thus, $G_0'$ has vertex-disjoint subgraphs $F_1,F_1'$ with
			$\chi(F_1),\chi(F_1')\ge \chi(G_0')-m\ge 2m$.
			By~\cref{thm:chi-insep} again,
			$G_1'$ has vertex-disjoint subgraphs $F_2,F_3$ with
			$\chi(F_2),\chi(F_3)\ge\chi(F_1')-m
			\ge \chi(G_0')-2m$.
			Put $a':=(\frac{2}{3})^{s+1}\vep'=\frac{2}{3}a$; and observe that for every $j\in\{1,2,3\}$,
			\[\chi(F_j)\ge\chi(G_0')-2m
			\ge Ct+(9C-39-64C_{\ref{thm:chi-insep}})a\cdot f(G,\vep,t)
			\ge 4Ca',\]
			so by~\cref{thm:kappachi}, $F_j$ has a $Ca'$-connected subgraph $L_j$ with
			(note that $\frac{3}{2}(9C-29-64C_{\ref{thm:chi-insep}})-2C\ge 11C$)
			\[\chi(L_j)\ge \chi(F_j)-2Ca'
			\ge Ct+(9C-29-64C_{\ref{thm:chi-insep}})a\cdot f(G,\vep,t)-2Ca'
			\ge C(t+11a'\cdot f(G,\vep,t)).\]
			Hence, by induction, $L_j$ is $((\frac{3}{4})^{s+1}\vep',a',3a')$-woven;
			let $V_j=\{v_{(j-1)a'+1},v_{(j-1)a'+2},\ldots,v_{ja'}\}\subset V(L_j)$.
			Since $\kappa(G_0)\ge (C-1)a\ge 20a$,
			$G_0$ is $2a$-linked by~\cref{thm:linked};
			so there is a linkage $\mac P'$ in $G_0$ linking $\{(u_i,v_i):i\in[2a]\}$.
			By~\cref{lem:woven} applied iteratively to each of $J_1,J_2,J_3$ (note that $3a'=2a$),
			there is a linkage $\mac P=\{P_1,\ldots,P_{2a}\}$ linking $\{(u_i,v_i):i\in[2a]\}$ in $G_0$
			and $((\frac{3}{4})^{s+1}\vep',a')$-dense graphs $H_1,H_2,H_3$ such that for each $j\in\{1,2,3\}$,
			there is an $H_j$ model $\mac M_j$ rooted at $V_j$ in $L_j$
			such that $V(\mac M_j)\cap V(\mac P)=V_j$.
			
			Now, for every $i\in[2a]$, let $j\in\{1,2,3\}$ be such that $v_i\in V_j$,
			and let $M_i$ be the member of $\mac M_j$ containing $v_i$.
			For every $i\in[a]$, let
			\[D_i:=M_i\cup P_i\cup r_i'u_i\cup r_i'u_{a+i}\cup P_{a+i}\cup M_{a+i},\]
			then it is not hard to see that $\{D_1,\ldots,D_a\}$ is an $H$ model rooted at $R'$ in $G'$ for some graph $H$ with
			\[\edge(\overline{H})
			\le\edge(\overline{H_1})
			+\edge(\overline{H_2})
			+\edge(\overline{H_3})
			\le 3\left(\frac{3}{4}\right)^{s+1}\vep'{a'\choose2}
			\le \left(\frac{3}{4}\right)^s\vep'{a\choose2},\]
			where the last inequality holds because
			$\frac{9}{4}{a'\choose2}=\frac{1}{8}(3a')(3a'-3)
			<\frac{1}{8}(2a)(2a-2)={a\choose2}$.
			Therefore $H$ is $((\frac{3}{4})^s\vep',a)$-dense,
			proving~\cref{claim:woven}.
		\end{subproof}
		The rest of the proof is now straightforward. 
		Recall that $\mac P_1$ is a set of $14a$ paths between $Z$ and $V(F_0)$ in $G$ which pairwise share no vertex in $G_1=G\setminus Z$,
		such that each vertex of $Z$ is the endpoint of exactly two paths in $\mac P_1$.
		Since $\kappa(G_1)\ge\kappa(G)-\abs{Z}\ge(C-7)a\ge 2a$,
		\cref{thm:menger} yields a set $\mac P_2$ of $2a$ vertex-disjoint paths between $V(G')$ and $V(F_0)$ in $G_1$.
		By~\cref{lem:redundant} applied to $\mac P_1,\mac P_2$ and by identifying in pairs the endpoints in $V(G')$ of the paths of $\mac P_2$,
		we obtain a set $\mac Q$ of $\abs{Z}+a=8a$ vertex-disjoint paths between $Z\cup V(G')$ and $V(F_0)$
		such that each vertex of $Z$ is an endpoint of some path of $\mac Q$.
		For each $i\in[a]$, let $x_i$ be the vertex in $V(F_0)$ linked to $r_i$ by some path $Q_i\in\mac Q$;
		and for each $i\in[b]$ (recall that $b=3a$),
		let $y_i,z_i$, respectively, be the vertices in $V(F_0)$ linked to $s_i,t_i$ by some paths $P_i,P_i'\in\mac Q$.
		Let $\{r_1',\ldots,r_a'\}$ be the set of endpoints in $V(G')$ of the paths of $\mac Q$;
		and for every $i\in[a]$,
		let $w_i$ be the vertex in $V(F_0)$ linked to $r_i'$ by some path $P_i'\in\mac Q$.
		Since $\kappa(F_0)\ge 40a$,
		$F_0$ is $4a$-linked by \cref{thm:linked};
		so there is a linkage $\mac N=\{N_1,\ldots,N_{4a}\}$ linking 
		$\{(x_i,w_i):i\in[a]\}\cup\{(y_i,z_i):i\in[b]\}$ in $F_0$.
		By~\cref{claim:woven},
		$G'$ has an $H$ model $\{M_1,\ldots,M_a\}$ rooted at $\{r_1',\ldots,r_a'\}$ for some
		$((\frac{3}{4})^s\vep',a)$-dense graph $H$.
		Now, let
		\begin{itemize}
			\item $A_i:=(Q_i\cup N_i)\cup (Q_i'\cup M_i)$ for every $i\in[a]$; and
			
			\item $B_{i}:=P_i\cup N_{a+i}\cup P_i'$ for every $i\in[b]$.
		\end{itemize}
		Then $\mac A=\{A_1,\ldots,A_a\}$ is an $H$ model rooted at $R$ in $G$
		and $\mac B=\{B_1,\ldots,B_b\}$ is a linkage in $G$ linking $\{(s_i,t_i):i\in[b]\}$
		such that $V(\mac A)\cap V(\mac B)=\emptyset$.
		This proves~\cref{lem:chisepwoven}.
	\end{proof}
	We are now ready to prove \cref{thm:main},
	which we restate here (in a slightly different from) for the convenience of the readers.
	\begin{theorem}
		\label{thm:reducing}
		There is an integer $C>0$ such that the following holds.
		Let $\vep\in(0,\frac{1}{256})$, and let $\Gamma:=\log(1/\vep)$.
		Let $t\ge4\Gamma^2$ be an integer, let $G$ be a graph, and let
		\[f(G,\vep,t)
		:=1+\max_{F\subset G}
		\left\{\frac{\chi(F)}{a}:a\ge(2\Gamma)^{-1/2}t,\,\abs{F}\le C\Gamma^4a,\,\text{$F$ is $(\vep,a)$-dense-minor-free}\right\}.\]
		If $\chi(G)\ge Ct\cdot f(G,\vep,t)$,
		then $G$ has a $(\vep,t)$-dense minor.
	\end{theorem}
	\begin{proof}
		Let $C:=80C_{\ref{lem:chisepwoven}}$.
		Let $\vep'\in(0,\frac{1}{256})$ satisfy $\vep=\vep'\log(1/\vep')$;
		then
		$\Gamma=\log(1/\vep')-\log\log(1/\vep')$,
		and so $\Gamma\le\log(1/\vep')\le2\Gamma$.
		Thus $t\ge 4\Gamma^2\ge\log^2(1/\vep')$.
		Let $t':=3^{\ceil{\log_3t}}$;
		then $t\le t'\le 3t$,
		and so
		\[f(G,\vep,t)\ge f_{\ref{lem:chisepwoven}}(G,\vep',t)
		\ge f_{\ref{lem:chisepwoven}}(G,\vep',t')\ge1.\]
		Hence $\chi(G)\ge Ct\cdot f(G,\vep, t)\ge 4C_{\ref{lem:chisepwoven}}t'$;
		so by~\cref{thm:kappachi}, $G$ has a $C_{\ref{lem:chisepwoven}}t'$-connected subgraph $G'$ with
		\[\chi(G')\ge \chi(G)-2C_{\ref{lem:chisepwoven}}t'
		\ge Ct\cdot f(G,\vep,t)
		-2C_{\ref{lem:chisepwoven}}t
		\ge 24C_{\ref{lem:chisepwoven}}t'\cdot f(G,\vep',t').\]
		By~\cref{lem:chisepwoven} with $s=0$
		and by the choice of $\vep'$,
		$G'$ is $(\vep,t',3t')$-woven,
		and so contains a $(\vep,t')$-dense minor.
		By averaging, $G'$ contains a $(\vep,t)$-dense minor, and so does $G$.
		This proves~\cref{thm:reducing}.
	\end{proof}
	\section{Additional remarks}
	An \emph{odd minor} of a graph $G$ is a graph obtained from a subgraph of $G$ by a sequence of edge cut contractions;
	and so every odd minor of $G$ is a minor of $G$.
	Odd Hadwiger's conjecture,
	a strengthening of the original one suggested by Gerards and Seymour (see~\cite[Section 6.5]{MR1304254}),
	says that every graph with chromatic number at least $t$ has an odd $K_t$ minor,
	and is still open.
	There is also the odd variant of linear Hadwiger's conjecture which is open as well,
	which asserts that for some universal constant $C>0$, every graph with chromatic number at least $Ct$ contains an odd $K_t$ minor.
	Linear Hadwiger's conjecture and its odd variant are recently shown to be equivalent by Steiner~\cite{MR4379303},
	who actually proved the following stronger result via an elegant argument.
	\begin{theorem}
		\label{thm:steiner}
		Let $H$ be a graph, and let $m>0$ be such that every graph with no $H$ minor has chromatic number at most $m$.
		Then every graph with no odd $H$ minor has chromatic number at most $2m$.
	\end{theorem}
	By~\cref{thm:steiner},~\cref{thm:chi} is equivalent to its odd version, as follows.
	\begin{theorem}
		\label{thm:chi2}
		There is an integer $C>0$ such that for every $\vep\in(0,\frac{1}{256})$ and every integer $t\ge2$,
		every graph with chromatic number at least $Ct\log\log(1/\vep)$ contains an odd $(\vep,t)$-dense minor.
	\end{theorem}
	\hypersetup{bookmarksdepth=-1}
	\subsection*{Acknowledgements}
	This work was partially done at the workshop \dd Seymour is $70+2\vep$\ee{} at LIP, ENS de Lyon in June, 2022.
	The author would like to thank the participants for creating a hospitable working environment,
	and to thank Paul Seymour, in particular, for helpful remarks and encouragement.
	He would also like to thank Luke Postle for useful discussions.
	\hypersetup{bookmarksdepth}

\end{document}